\documentclass[12pt,letterpaper]{amsart}

\usepackage{macros}
\usepackage{ulem}
\standardmargins\standardpackages\standardrenews\standardlabeling\standardmakes
\colorcommentstrue

\newcommand{\noverm}{{n/m}}
\newcommand{\doubleast}{\ast\ast}

\renewcommand{\emph}[1]{{\it#1}}

\newcommand{\Rp}{\R_{+}}
\renewcommand{\det}{\mathrm{det}}

\newcommand{\tbase}{s}
\newcommand{\tmod}{x}

\newcommand{\point}{\aa}
\newcommand{\pointtbase}{\point_{\tbase}}
\newcommand{\Ltbase}{L_{\tbase}}
\newcommand{\SQnew}{S'_Q}
\newcommand{\SQorig}{S''_Q}
\newcommand{\Bad}{\mathbf{Bad}}
\newcommand{\Twist}{\mathcal{T}}
\newcommand{\TTwist}{\widetilde\Twist}
\newcommand{\TS}{\widetilde{S}}
\newcommand{\Tad}{\mathcal{T}\!\mathbf{ad}}
\newcommand{\LambdaQ}{\Lambda_Q}
\begin{document}
\title{Twisted Diophantine approximation on manifolds}

\author{Victor Beresnevich}
\address{Department of Mathematics, University of York, Heslington, York YO10 5DD, UK}
\email{victor.beresnevich\at york.ac.uk}

\author{David Simmons}
\address{434 Hanover Ln, Irving, TX 75062, USA}
\email{david9550\at gmail.com}

\author{Sanju Velani}
\address{Department of Mathematics, University of York, Heslington, York YO10 5DD, UK}
\email{sanju.velani\at york.ac.uk}


\dedicatory{Dedicated to Barak Weiss - 60 not out!}

\begin{Abstract}
In twisted Diophantine approximation, for a fixed  $m\times n$ matrix $\boldsymbol\alpha$ one is interested in sets of vectors $\boldsymbol\beta\in\mathbb R^m$ such that the system of affine forms $\mathbb R^n \ni \mathbf q \mapsto \boldsymbol\alpha\mathbf q + \boldsymbol\beta \in \mathbb R^m$ satisfies some given Diophantine condition. In this paper we introduce the notion of manifolds which are of $\boldsymbol\alpha$-twisted Khintchine type for convergence or divergence. We provide sufficient conditions under which nondegenerate analytic manifolds exhibit this twisted Khintchine-type behaviour. Furthermore, we investigate the intersection properties of the sets of  $\boldsymbol\alpha$-twisted badly approximable and well approximable vectors with nondegenerate manifolds.
\end{Abstract}
\maketitle

\section{Introduction}

Fix $\pdim,\qdim\in\N$ and let $\psi: \Rp \to\Rp$ be a nonincreasing continuous function, where $\Rp:=(0,\infty)$; we will call such a function an \emph{approximation function}. Let $\bfalpha = (\alpha_{ij})$ be an $m\times n$ matrix over $\R$ and  $\bfbeta = (\beta_1,\ldots,\beta_m)^{T}\in\R^m$. The system of affine forms
\begin{equation}\label{vb1}
\R^n \ni \qq = (q_1,\ldots,q_n)^{T} \mapsto \bfalpha\qq + \bfbeta = (\alpha_{i1} q_1 + \ldots + \alpha_{in} q_n + \beta_i)_{1 \leq i \leq m} \in \R^m\,,
\end{equation}
where $\qq$ and $\bfbeta$ are column vectors,
is called \emph{$\psi$-approximable} if there exist infinitely many pairs $(\pp,\qq)\in \Z^\pdim\times (\Z^\qdim\butnot \{\0\})$ such that
\begin{equation}\label{vb1.2}
\|\bfalpha\qq + \bfbeta + \pp\| \leq \psi(\|\qq\|).
\end{equation}
Here and elsewhere $\|\cdot\|$ denotes the Euclidean norm.

The set of all pairs $(\bfalpha,\bfbeta)$ such that the system of affine forms \eqref{vb1} is $\psi$-approximable will be denoted by $\WW_\psi$.
By fixing $\bfbeta$ and letting $\bfalpha$ vary, one winds up in the familiar realm of \emph{inhomogeneous Diophantine approximation}; the special case $\bfbeta = \0$ is called \emph{homogeneous Diophantine approximation}. By contrast, if we fix $\bfalpha$ and let $\bfbeta$ vary then the setup is called \emph{twisted Diophantine approximation} \cite{MR2508636,BHKV}.

In this paper we consider \emph{twisted Diophantine approximation on manifolds}, in which, after $\bfalpha \in \R^{m \times n}$ is fixed, the vector $\bfbeta$ is restricted to lie in some submanifold of $\R^m$.  This contrasts the classical  theory of \emph{Diophantine approximation on manifolds} \cite{MR3618787,BernikDodson,MR548467} in which $\bfbeta \in \R^m$ is fixed and $\bfalpha$ is restricted to lie in some submanifold of $\R^{m \times n}$.

In the standard, classical theory of Diophantine approximation on manifolds, a common problem is to attempt to determine the behaviour of almost every point on a manifold with respect to the induced Lebesgue measure on the manifold (see Remark~\ref{svvb}).    In particular, a manifold $\MM$ is said to be of \emph{Khintchine type for divergence (resp. convergence)} if for every approximation function $\psi$ such that the series
\begin{equation}
\label{psiseries}
\sum_{q=1}^\infty q^{\qdim - 1} \psi^\pdim(q)
\end{equation}
diverges (resp. converges), almost every (resp. almost no) point of $\MM$ is  $\psi$-approximable. More precisely, a point $\bfalpha\in\MM$ is called inhomogeneously $\psi$-approximable if $\bfalpha \in \WW_\psi(\bfbeta)$ for every $ \bfbeta$ where
$$
\WW_\psi(\bfbeta):=\{\bfalpha\in\R^{m \times n} :(\bfalpha,\bfbeta) \in \WW_\psi\} \,
$$
and homogeneously $\psi$-approximable if $\bfalpha \in \WW_\psi(\bf 0)$. Clearly, an inhomogeneous Khintchine type result contains the corresponding homogeneous result.  The terminology is adopted
to emphasize  Khintchine's contribution who discovered these properties in the case $\MM = \R^m$.

When $m=1$, we are in the framework of the so called dual theory, and in the homogeneous case it was shown that every nondegenerate\Footnote{If $\ff:\R^d\supset U \to \MM \subset \R^m$ is a local parameterization of a manifold $\MM$, then $\MM$ is called \emph{nondegenerate} at a point $\ff(\xx)$ if for some $k\in\N$, $\ff$ is $\CC^k$ and the partial derivatives of $\ff$ at $\xx$ (including higher-order partial derivatives up to order $k$) span $\R^m$ \cite{KleinbockMargulis2}. The manifold $\MM$ is called nondegenerate if it is nondegenerate at almost every point. If $\MM$ is connected and analytic, then $\MM$ is nondegenerate if and only if it is not contained in any affine hyperplane.} manifold is of  Khintchine type for both convergence \cite{Beresnevich_Groshev,BKM2} and divergence \cite{BBKM}. The inhomogeneous generalizations  of these results were subsequently obtained in \cite{BBV}.

When $n=1$, we are in the framework of the so called simultaneous theory, and  it is was shown that every analytic manifold is of homogeneous  Khintchine type for divergence \cite{Beresnevich_Khinchin} and every nondegenerate curve is of inhomogeneous Khintchine type for divergence \cite{BVVZ2}.
Very recently, it was shown in \cite{BeresnevichDatta1} that every nondegenerate manifold is of Khintchine type for divergence. In the convergence case, within the homogeneous setting, until recently only partial results had been known \cite{DRV,BVVZ,Simmons10} before it was verified in \cite{Beresnevich-Yang} that  every nondegenerate manifold is of Khintchine type for convergence.  The latter has also been recently generalised to the inhomogeneous setting \cite{BeresnevichDatta1}.

The upshot is that the standard Khintchine type theory is today in a reasonably complete state for nondegenerate manifolds.  Surprisingly nothing seems to be known for twisted Diophantine approximation on manifolds. The purpose of this work is to address this imbalance. In what follows
$$
\Twist_\psi(\bfalpha):=\{\bfbeta\in\R^m:(\bfalpha,\bfbeta) \in \WW_\psi\}$$ is the set of $(\psi,\bfalpha)$-approximable
points $\bfbeta\in\R^m$. Our goal is to establish Khintchine type results for the intersections of $\Twist_\psi(\bfalpha)$ with proper submanifolds $\MM$ of $\R^m$.  Note that it makes sense to assume that $m\geq 2$, since the only submanifolds of $\R^1$ are open sets and discrete sets.

To begin with, we consider the non-manifold theory, that is, when $\MM=\R^m$.
Underpinning this theory is the following fundamental result due to Kurzweil \cite{Kurzweil}.

\begin{theorem}[Kurzweil, {\cite{Kurzweil}}] We have that
\begin{equation}\label{Kurz}
\text{\it$\bfalpha$ is badly approximable\quad\iff\quad} \begin{minipage}{0.45\textwidth}for any nonincreasing function $\psi$ such that \eqref{psiseries}  diverges, $\Twist_\psi(\bfalpha)$ has full Lebesgue measure.
\end{minipage}
\end{equation}
\end{theorem}

Recall,  that $\bfalpha$  is said to be badly approximable if $\bfalpha \notin  \WW_{c\psi}(\bf0)$   for some $c>0$ and $\psi(q)= q^{-n/m}$. For further details of badly approximable points  see the discussion centred around \eqref{Bad} below. At this point it is enough to note that the set of such points is of zero measure.   Returning to  Kurzweil's theorem, a crucial observation to make  is that there is an interplay between properties imposed on $\bfalpha$ and the class of approximation functions $\psi$  considered. Indeed, if we wish to include a larger class of $\bfalpha$ on the left hand side of \eqref{Kurz}, we would have to tighten the condition on $\psi$ on the right of \eqref{Kurz}. For instance, Chaika and Constantine \cite{MR3941149} established that \eqref{Kurz} holds for $\bfalpha$ from a much larger set of full measure if we request that the functions $\psi$ in \eqref{Kurz} are such that $q\mapsto q\psi(q)$ is nonincreasing\footnote{For completeness, we note that when 
$m=n=1$, Fuchs and Kim \cite{MR3494133} extended Kurzweil’s theorem to arbitrary $\alpha$
by replacing the series \eqref{psiseries} with an appropriately defined series involving the principal convergents of $\alpha$.}. Note that this interplay is not  present in the  standard theory where  for any nonincreasing function $\psi$  such that \eqref{psiseries}  diverges, $\WW_\psi(\bfbeta)$ has full Lebesgue measure irrespective of $\bfbeta  \in \R^m$.

The upshot of the above discussion is that even for the  $\R^m$ setting, Khintchine type results within the twisted framework are dependent on the properties of $\bfalpha$ and $\psi$. Thus, in order to develop a twisted Khintchine type theory of Diophantine approximations on manifolds it is natural to introduce the following notion that brings into play manifolds and formalises the trade-off between the properties of $\bfalpha$ and the admissible class of approximation functions $\psi$. Given an $m\times n$ matrix $\bfalpha$ and a class of approximation functions $\FF$, we will say that a manifold $\MM\subset\R^m$ is of:
\begin{itemize}
    \item
\emph{$\bfalpha$-twisted Khintchine type for convergence over $\FF$} if
\begin{equation}\label{vb1.4}
\forall\;\psi\in\FF\quad\text{\eqref{psiseries} converges}\;\;\Rightarrow\;\;\text{$\bfbeta\in\Twist_\psi(\bfalpha)$ for almost every $\bfbeta\in\MM$}\,;
\end{equation}

\bigskip

\item \emph{$\bfalpha$-twisted Khintchine type for divergence over $\FF$} if \begin{equation}\label{vb1.5}
\forall\;\psi\in\FF\quad\text{\eqref{psiseries} diverges}\;\;\Rightarrow\;\;\text{$\bfbeta\not\in\Twist_\psi(\bfalpha)$ for almost every $\bfbeta\in\MM$}\,.
\end{equation}
\end{itemize}

\bigskip

In this paper we will be addressing the following general problem for nondegenerate manifolds. 

\medskip

\noindent\textbf{General Problem:} {\it Given a class $\FF$ of approximation functions $\psi$, under what conditions on $\bfalpha$ is it true that every nondegenerate manifold $\MM\subset\R^m$
is of $\bfalpha$-twisted Khintchine type for convergence and/or divergence over $\FF$?
}

\medskip


To place this general problem within the context  of the classical theory, it is natural to focus on specific classes of approximation functions.
As a starting point,  let $\FF_{\rm A}$   be the class of \textbf{all} nonincreasing functions $\psi$. Then the general problem reduces to  the following, which may be viewed as seeking an analogue of Kurzweil’s theorem for nondegenerate manifolds.

\medskip

\noindent\textbf{Problem A:} {\it
Determine conditions on $\bfalpha$ so that any nondegenerate manifold is of $\bfalpha$-twisted Khintchine type for convergence/divergence over $\FF_{\rm A}$.}

\medskip

In view of Kurzweil's theorem,  $\bfalpha$ in Problem~A should be at least badly approximable.  By restricting our attention to subclasses of  $\FF_{\rm A}$ we would expect to be able to  relax  the  condition  that $\bfalpha$ is badly approximation.  We now explore two natural subclasses.

\bigskip

Let $\FF_{\rm B}  \subset \FF_{\rm A}  $ denote the  class of approximation functions $\psi : \Rp\to\Rp$   consisting  of
$$
\psi \, :  \,    q    \ \to  \ \psi(q)\, :=  \, c \,  q^{-n/m}    \qquad {\rm \ for \  some \ constant  }  \quad  c>0   \, .
$$
This class is instrumental for defining the so-called badly approximable systems \cite{MR1724255}.
A pair $(\bfalpha,\bfbeta)$ such that $(\bfalpha,\bfbeta)\not\in\WW_\psi$ for some $\psi\in\FF_{\rm B}$ is called \emph{badly approximable}. The set of all badly approximable  pairs $(\bfalpha,\bfbeta)$ will be denoted by $\Bad$. Thus
\begin{equation}\label{Bad0}
\Bad = \bigcup_{\psi\in\FF_{\rm B}} \R^{m \times n}\butnot \WW_{\psi}.
\end{equation}
In the standard classical theory one studies the sets of inhomogeneously badly approximable points
\begin{equation}\label{Bad}
\Bad(\bfbeta):=\{\bfalpha:(\bfalpha,\bfbeta)\in\Bad\}\,.
\end{equation}
In view of the inhomogeneous Khintchine type result, we have that $\Bad(\bfbeta)
$ is of zero measure for all $\bfbeta \in \R^m$. However, it is well know that  it is a set of full Hausdorff dimension; that is,
$$ \dim \Bad(\bfbeta) = mn \, .  $$
For further details   see \cite{3problems,KW}  and references within.   We now turn our attention to the analogue of badly approximable within the twisted framework.  For a given $\bfalpha \in \R^{m \times n}$, the set of badly $\bfalpha$-approximable vectors is defined by
$$
\Tad(\bfalpha):=\{\bfbeta\in\R^m:(\bfalpha,\bfbeta)\in\Bad\}\,.
$$
Clearly a point $\bfbeta$ is badly $\bfalpha$-approximable if and only if it is not $(\psi,\bfalpha)$-approximable for some $\psi\in\FF_{\rm B}$. In other words,
\[
\Tad(\bfalpha) := \bigcup_{\psi\in\FF_{\rm B}} \R^\pdim\butnot \Twist_{\psi}(\bfalpha).
\]
The set of badly $\bfalpha$-approximable vectors is known to be of full Hausdorff dimension \cite{BFS1,BHKV}. It is also know that its  Lebesgue measure depends on the Diophantine properties of $\bfalpha$. This is not the case in classical framework of inhomogeneous Diophantine approximation in which the analogous set of inhomogeneous badly approximable vectors is always of measure zero. In this paper we will be interested in badly $\bfalpha$-approximable points on manifolds. In particular, we will address the following problem, which, in view of the fact that for any $\psi\in\FF_{\rm B}$ the sum \eqref{psiseries} diverges, is another special case of the General Problem.

\medskip

\noindent\textbf{Problem B:} {\it
Determine conditions on $\bfalpha$ so that for any nondegenerate manifold $\MM\subset\R^m$ almost every point $\bfbeta\in\MM$ is not badly $\bfalpha$-approximable.}

\bigskip

Let $\FF_{\rm E}  \subset \FF_{\rm A}  $ denote the  class of approximation functions $\psi : \Rp\to\Rp$   consisting  of
$$
\psi \, :  \,    q    \ \to  \ \psi(q)\, :=   \,  q^{-n/m - \epsilon}    \qquad {\rm \ for \  some \  }  \quad  \epsilon>0   \, .
$$ This  class is associated with the property of ``extremality'' in the standard Khintchine-type theory.
A pair $(\bfalpha,\bfbeta)$ such that $(\bfalpha,\bfbeta)\not\in\WW_\psi$ for some $\psi\in\FF_{\rm E}$ is called {\it very well approximable}. The set of all very well approximable  pairs $(\bfalpha,\bfbeta)$ will be denoted by $\WW$. Thus
\begin{equation}\label{Extremal0}
\WW = \bigcup_{\psi\in\FF_{\rm E}} \WW_{\psi}.
\end{equation}
In the standard classical theory one studies the sets of inhomogeneously very well approximable points
\begin{equation}\label{Extremal}
\WW(\bfbeta):=\{\bfalpha:(\bfalpha,\bfbeta)\in\WW\}\,.
\end{equation}
In view of the inhomogeneous Khintchine type result, we have that $\WW(\bfbeta)
$ is of zero measure for all $\bfbeta \in \R^m$.  We now turn our attention to the analogue of very well approximable within the twisted framework.  For a given $\bfalpha \in \R^{m \times n}$, the set of very well $\bfalpha$-approximable vectors is defined by
$$
\WW(\bfalpha):=\{\bfbeta\in\R^m:(\bfalpha,\bfbeta)\in\WW\}\,.
$$
The following is the specialisation of the General Problem to the class $\FF_{\rm E}$.

\medskip

\noindent\textbf{Problem C:} {\it
Determine conditions on $\bfalpha$ so that for any nondegenerate manifold $\MM\subset\R^m$ almost every point $\bfbeta\in\MM$ is not very well $\bfalpha$-approximable.}

\begin{remark}
The framework of twisted Diophantine approximations  concerning the sets $\Twist_\psi(\bfalpha)$ can be regarded as the quantitative refinement of Kronecker's classical theorem \cite[p.~53]{Cassels} on the density of the set $\boldsymbol\alpha\mathbb Z^n$ modulo $1$.
This set can be viewed as the orbit of the $\mathbb Z^n$-action
$$
\mathbb Z^n\times\mathbb T^m\to \mathbb T^m
$$
on the torus $\mathbb T^m=(\mathbb R/\mathbb Z)^m$, defined by
$$
(\mathbf{q},\boldsymbol{\beta})\mapsto\boldsymbol{\alpha}\mathbf{q}+\boldsymbol{\beta}\mod 1\,,
$$
see \cite{BHKV} for details.
In the case of $n=m=1$ it boils down to circle rotations by the angle $2\pi\alpha$ -- an area that has deep connections to the theory of continued fractions and that has been studied in  depth. For $n=1$ and $m>1$ it can be understood as torus rotations, and the resulting orbit is nothing but a Kronecker's sequence. Many problems have been studied in twisted Diophantine approximation including standard and weighted badly approximable points \cite{MR3630719, BMS, MR3628933}, well approximable points \cite{MR4834219,MR2335077, Moshchevitin5, Shapira}, multiplicative problems \cite{MR3717987, MR2753608},
Khintchine type results \cite{MR3494133, MR2853045, Kim2023, Kurzweil, Simmons8} and results with restricted $\qq$ \cite[Thm~1.11]{HK25}. We re-iterate that the main goal of this paper is to advance the twisted Khintchine type theory of Diophantine approximations on manifolds. To the best of our knowledge all previous twisted results are for $\R^m$ only.
\end{remark}

\subsection{Main results}

We need to introduce a little more terminology before stating our main theorems regarding manifolds within the twisted framework. A nonincreasing function $\psi:\Rp\to\Rp$ is called \emph{doubling} if there is $K>1$ such that
$$
\psi(x)\leq K\psi(2x)\qquad\text{for all $x>0$.}
$$
It is readily seen that a nonincreasing function $\psi:\Rp\to\Rp$ is doubling if and only if for every $K_1 > 1$, there exists $K_2 > 1$ such that
\[
K_1^{-1} \leq y/x \leq K_1  \ \  \text{ implies } \  \  K_2^{-1} \leq \psi(y)/\psi(x) \leq K_2.
\]
Indeed, one can take $K_2=(2K_1)^{\log_2 K}$.
Next, for each $\tau > 0$, write
$$\psi_\tau(q) := q^{-\tau}  \, . $$ The \emph{exponent of irrationality} of a matrix $\bfalpha$ is the supremum of $\tau$ such that $\bfalpha$ is homogeneously $\psi_\tau$-approximable:
\[
\omega(\bfalpha) \df \sup\{\tau > 0 : (\bfalpha, \0 ) \in \WW_{\psi_\tau}\}.
\]
It follows from the well-known theorems of Dirichlet and Khintchine that every $m\times n$ matrix $\bfalpha$ satisfies $\omega(\bfalpha) \geq n/m$, and that almost every matrix $\bfalpha$ satisfies $\omega(\bfalpha) = n/m$. The matrix $\bfalpha$ is called \emph{very well approximable} if $\omega(\bfalpha) > \qdim/\pdim$; thus almost every matrix is not very well approximable.

We can now state our main result regarding manifolds of $\bfalpha$-twisted Khintchine type for convergence. In all of the theorems below, we assume that $m\geq 2$.

\begin{theorem}[Khintchine type for convergence]
\label{theorem1}
Let $\bfalpha$ be an $m\times n$ matrix whose transpose $\bfalpha^T$ satisfies
\begin{equation}
\label{NVWA}
\omega(\bfalpha^T) < \left(\frac\qdim\pdim - \frac\qdim{2\pdim^2(\pdim-1)}\right)^{-1}.
\end{equation}
Then for any doubling approximation function $\psi$ such that \eqref{psiseries} converges, $\Twist_\psi(\bfalpha)$ has zero measure on any nondegenerate curve and on any nondegenerate analytic manifold. In other words, every nondegenerate analytic manifold is of $\bfalpha$-twisted Khintchine type for convergence over the class of doubling approximation functions.
\end{theorem}

\medskip

\begin{remark}
Note that by Khintchine's transference principle, if $\bfalpha$ is not very well approximable then \eqref{NVWA} holds.
\end{remark}

\begin{remark}
We have restricted our attention to analytic manifolds for convenience, as was done in \cite{Beresnevich_BA}. The reason for this is that we use a Fibering Lemma from \cite{Beresnevich_BA} which was proven for analytic manifolds. Although the proof can be generalized to a larger class of manifolds (e.g. $\CC^\infty$), it is not known what the weakest conditions are that guarantee that such a fibering lemma holds.
\end{remark}

Regarding manifolds of $\bfalpha$-twisted Khintchine type for divergence, we need to introduce a stronger assumption on the approximation functions. A \emph{Hardy $L$-function} is a function that can be expressed using only the elementary arithmetic operations $+,-,\times,\div$, exponents, logarithms, and real-valued constants, and that is well-defined on some interval of the form $(t_0,\infty)$. In what follows the function $\log^{(j)}(q)$ will denote the $j$th iterate of the natural logarithm, which is thus defined and positive on $(t_0,\infty)$ for a sufficiently large $t_0=t_0(j)$.

\def\psii{h_i}

\begin{theorem}[Khintchine type for divergence]
\label{theorem2}
Fix $i\geq 1$  and let $\bfalpha$ be an $m\times n$ matrix which is not $\phi_i$-approximable, where
\begin{align}
\label{phidef}
\phi_i(q) = \psii(q) \log\log(q) \quad {\rm and}  \quad
\psii(q): = \frac{1}{(q^\qdim\prod_{j=1}^i\log^{(j)}(q))^{1/\pdim}}\,.
\end{align}
Then if $\psi$ is a Hardy $L$-function such that \eqref{psiseries} diverges, then $\Twist_\psi(\bfalpha)$ has full measure on any nondegenerate analytic manifold. In other words, every nondegenerate analytic manifold is of $\bfalpha$-twisted Khintchine type for divergence over the class of Hardy $L$-functions.
\end{theorem}

\begin{remark}
Note in particular that Theorem~\ref{theorem2} applies when $\bfalpha$ is (homogeneously) badly approximable, i.e. when $(\bfalpha,\0) \notin \WW_{c\psi_\noverm}$ for some $c > 0$, where $$
\psi_\noverm(q) \,  =   \, q^{-n/m} \, . $$ On the other hand, by Khintchine's theorem, the set of $\bfalpha$ to which the theorem applies is a Lebesgue nullset, since $\sum_{q=1}^\infty q^{\qdim-1} \phi_i^\pdim(q) = \infty$ for any $i \in \N$.
\end{remark}

We finish with some results regarding  the twisted set of badly approximable vectors. With this in mind, recall that
an $m\times n$ matrix $\bfalpha$ is called \emph{singular} if for all $\epsilon > 0$, there exists $Q_\epsilon \geq 1$ such that for all $Q\geq Q_\epsilon$, there exists $(\pp,\qq)\in\Z^m\times \Z^n$ such that
\[
\|\bfalpha\qq + \pp\| \leq \epsilon Q^{-n/m} \text{ and } 0 < \|\qq\| \leq Q.
\]
The following theorem implies that if $\bfalpha$ is nonsingular, then almost every vector on a nondegenerate analytic manifold is not badly $\bfalpha$-approximable.

\begin{theorem}[Measure of badly $\bfalpha$-approximable points, nonsingular case]
\label{theorem3}
Let $\bfalpha$ be a nonsingular $m\times n$ matrix. Then the set of badly $\bfalpha$-approximable vectors has zero measure on any nondegenerate analytic manifold. In other words, for all $c > 0$, if $\psi = c \psi_\noverm$ then $\Twist_\psi(\bfalpha)$ has full measure on any nondegenerate analytic manifold.
\end{theorem}

Note that in the case where the nondegenerate analytic manifold is the whole space, Theorem~\ref{theorem3} follows from \cite[Theorem 3.4]{Shapira} due to Shapira. More recently Moshchevitin \cite{Moshchevitin5} has proven a generalization of Shapira's result in the setting where $n=1$ and $\bfalpha_1,\bfalpha_2,\ldots$ is a ``well-distributed'' sequence of points in $\R^m$, showing that $\liminf_{k\to\infty} k^{1/m} \min_{\pp\in\Z^m}\|\bfalpha_k - \pp - \bfbeta\| = 0$ for Lebesgue-a.e. $\bfbeta\in \R^m$. This yields Shapira's result in the special case where $\bfalpha_k = k \bfalpha$. Even more recently Taehyeong Kim \cite{Kim2023} has re-proved Shapira's result and established a quantitative improvement using ubiquity. We would like to stress that all of these results are for the whole space $\R^\pdim$, whereas our results are for arbitrary nondegenerate analytic manifolds.

By contrast, if we strengthen the requirement of singularity slightly, then almost every vector on a nondegenerate analytic manifold  becomes badly $\bfalpha$-approximable. Specifically, an $m\times n$ matrix $\bfalpha$ is called \emph{very singular} if there exist $\epsilon > 0$ and $Q_\epsilon \geq 1$ such that for all $Q\geq Q_\epsilon$, there exists $(\pp,\qq)\in\Z^m\times\Z^n$ such that
\begin{equation}
\label{verysingulardef}
\|\bfalpha\qq + \pp \| \leq Q^{-(n/m + \epsilon)} \text{ and } 0 < \|\qq\| \leq Q.
\end{equation}
The Hausdorff dimension of the very singular matrices is the same as the Hausdorff dimension of the singular matrices \cite{DFSU_singular}.

\begin{theorem}[Measure of badly $\bfalpha$-approximable points, very singular case]
\label{theorem5}
Let $\bfalpha$ be a very singular $m\times n$ matrix. Then the set of badly $\bfalpha$-approximable vectors has full measure on any nondegenerate analytic manifold.
\end{theorem}
In the non-manifold case, Theorem~\ref{theorem5}  can be deduced as a consequence of the main theorem in \cite{BugeaudLaurent05} and the  Khintchine type transference inequalities of Jarn\'ik and Apfelbeck, see \cite[Equation~(6)]{BugeaudLaurent05}.

\medskip
\begin{remark}
It is worth mentioning that our theorems essentially resolve Problem~B and makes progress towards Problems~A and C.
In particular, since $\FF_{\rm E}$ is a subclass of the class of doubling functions,  Theorem~\ref{theorem1}  answers Problem C for any $\bfalpha$ subject to \eqref{NVWA}. Note that the set of  such $\bfalpha$ is  a significantly larger set than that of badly approximable points required within the context of Problem~A. Indeed, the former is of full measure while the latter is null.
\end{remark}
\medskip

Finally, rather than asking about the measure of the set of badly $\bfalpha$-approximable vectors, we can ask about its Hausdorff dimension. In this direction we have the following result.

\begin{theorem}[Dimension of badly $\bfalpha$-approximable points]
\label{theorem4}
Let $\bfalpha$ be a badly approximable $m\times n$ matrix. Then the set of badly $\bfalpha$-approximable vectors is absolute winning and thus has full dimension on any $C^1$ manifold.
\end{theorem}

Here ``absolute winning'' means winning for McMullen's ``absolute game'', see Appendix~\ref{sectionappendix}.  The theorem  strengthens the  result of Bengoechea, Moshchevitin, and Stepanova \cite{BMS} who proved the statement  for Schmidt's game rather than the absolute game.  Their proof can easily be adapted to prove the stronger absolute winning result. For completeness, we give a proof of Theorem~\ref{theorem4}  in Appendix~\ref{sectionappendix}.

\medskip

\begin{remark}
We emphasize that, in Theorem~\ref{theorem4}, the manifolds are not required to be nondegenerate, in contrast with the assumptions imposed in the previous results. In fact, the framework of the General Problem and its associated subproblems naturally extends to degenerate manifolds, including affine subspaces. We also note that nondegeneracy is not a prerequisite for progress in a number of problems in the classical (non-twisted) setting. For instance, it is not required for establishing a lower bound on the dimension of very well approximable points obtained in \cite{MR3731303}, which holds for all $C^2$ manifolds. Moreover, the classical theory encompasses affine subspaces (which are necessarily degenerate) satisfying suitable Diophantine conditions, as well as manifolds that are nondegenerate relative to such subspaces; see \cite{MR4120305,MR4706444,Kleinbock1} and the references within.
These observations raise the fundamental question of the precise role played by nondegeneracy in the twisted setting. Addressing this question appears to be a natural and interesting direction for further investigation. In our proofs, nondegeneracy manifests itself through the measure estimate \eqref{estimatesum}, the analysis of which leads to a specific line of investigation. For degenerate manifolds, this estimate will differ substantially, requiring a fundamentally different analytical approach. We hope to return to this problem in future work.
\end{remark}

\section{Preliminaries}
\label{prelim}

To begin with, we introduce various pieces of useful and relatively standard notation.

\medskip

\begin{notation}~ \label{not1}
\begin{itemize}
\item $A \lesssim B$ means that there exists a constant $K$ (called the \emph{implied constant}) such that $A \leq KB$, and $A \asymp B$ means that $A \lesssim B \lesssim A$.
\item $A+B$ denotes the Minkowski sum of two sets $A$ and $B$, i.e.
\[
A + B = \{a + b : a\in A, \; b\in B\}.
\]
Similarly, if $A \subset \R$ and $R\subset \R^d$, then $AR = \{t\vv : t\in A, \vv\in R\} \subset \R^d$.
\item If $R_1$ and $R_2$ are convex, centrally symmetric regions then $R_1 \lesssim R_2$ means that there exists a constant $C$ such that $R_1 \subset C R_2$.
\item If $R\subset \R^d$ is a convex, centrally symmetric region then its \emph{polar region} is
\[
R^* = \{\dd\in \R^d : |\xx\cdot\dd| \leq 1 \all \xx\in R\}
\]
and if $\Lambda \leq \R^d$ is a lattice then its \emph{polar lattice} is
\[
\Lambda^* = \{\dd\in\R^d : \xx\cdot\dd\in \Z \all \xx\in \Lambda\}.
\]
\item $\lambda_i(R;\Lambda)$ denotes the $i$th Minkowski minimum of a convex, centrally symmetric region $R$ with respect to $\Lambda$. Note that the duality principle for Minkowski minima states that $\lambda_i(R^*,\Lambda^*) \asymp \lambda_{d+1-i}(R,\Lambda)$ \cite[Theorem VIII.5.VI]{Cassels3}.

\item \smallskip
$|A|$ denotes the Lebesgue measure of a set $A$. If $A$ is a subset of a manifold in $\R^m$, $|A|$ denotes the Hausdorff measure of $A$ in the dimension of the manifold. Equivalently, $|A|$ denotes the pushforward Lebesgue measure of $A$ with respect to some fixed parameterisation.
\end{itemize}
\end{notation}

The following remarks allow us to make some useful additional assumptions when proving the main theorems.

\begin{remark}
\label{remark1}
In the proofs of Theorems \ref{theorem1}--\ref{theorem3}, we can without loss of generality assume that the manifolds in question are curves. Indeed, suppose that the theorems are true for curves, and let $M$ be a nondegenerate analytic manifold. By \cite[The Fibering Lemma]{Beresnevich_BA}, every element of $M$ has a neighbourhood that can be fibered as the disjoint union of nondegenerate analytic curves. By assumption, the conclusions of the theorems are true for almost every point on each of these curves, and thus by Fubini's theorem they are true for almost every point in the neighbourhood under consideration. By covering the manifold with such neighbourhoods, one sees that almost every point on the manifold has the desired properties.

We can make a further reduction as follows. Let $\ff:I_0\to \CC \subset \R^m$ be a parameterization of a nondegenerate curve, where $I_0$ is a closed interval. By definition, this means that $\CC$ is nondegenerate at almost every point. Therefore, for the purpose of the proofs of Theorems \ref{theorem1}--\ref{theorem3} we can assume without loss of generality that $\CC$ is nondegenerate at every point. Furthermore,  the Wronskian of $\ff'$,
\[
\det[\ff'(t),\dots, \ff^{(m)}(t)],
\]
is non-zero except at a finite set of points. The latter can be seen, for example, as a simple modification of \cite[Lemma~3]{MR1387861}.  Since we are looking to establish ``for almost everywhere''  results, we can ignore small neighbourhoods of these singularities and assume that the Wronskian of $\ff'$ is bounded from below by a fixed constant. We can also assume without loss of generality that $\ff^{(k)}$ is bounded on $I_0$ for every $1\le k\le m$.

We note that the assumption of analyticity is only used in the application of The Fibering Lemma from  \cite{Beresnevich_BA} and otherwise is not necessary.
\end{remark}

\begin{remark}  \label{svvb}
If
$\ff:I_0\to \CC \subset \R^m$ is the  parametrization of a curve $\CC=\CC_{\ff}$, and $A\subset\CC$, then the pushforward Lebesgue measure of $A$ will be the Lebesgue measure of $\ff^{-1}\big(A\big)$. In particular the pushforward Lebesgue measure of the set of interest
$$ \CC_{\ff}  \cap  \Twist_\psi(\bfalpha) $$
will be understood as the Lebesgue measure of $\ff^{-1}\big(\Twist_\psi(\bfalpha)\big) $. Note that despite the fact that the pushforward Lebesgue measure depends on the choice of the parametrization $\ff$, the null sets are independent of the choice. In other words, the measures arising from the different parametrization choices are equivalent.
\end{remark}

\begin{remark}\label{remarkij}
In the proof of Theorem~\ref{theorem2}, we can without loss of generality suppose that $\psi = \psii$. Indeed, let $\psi$ be a Hardy $L$-function such that \eqref{psiseries} diverges. Then upon using standard facts concerning  Hardy $L$-functions, it can be shown that there exists $j$ such that $\psi \gtrsim \psi_j$. To see this, let $j$ be the order of $\psi$ as a Hardy $L$-function (see e.g. \cite[p.24]{Hardy} for the definition). Then $\psi/\psi_j$ is of order $\leq j$, and thus by \cite[Theorem 3]{Hardy2}, we have either $\psi/\psi_j \gtrsim 1$ or else $\psi(q)/\psi_j(q) \lesssim (\log^{(j)}(q))^{-\delta}$ for some $\delta > 0$. In the latter case \eqref{psiseries} converges, so the former case holds and thus $\psi \gtrsim \psi_j$. By increasing $i$ or $j$ as necessary, we may without loss of generality suppose that $i = j$.
\end{remark}

\begin{remark}
\label{remark2}
In the proofs of Theorems \ref{theorem1}--\ref{theorem3}, we can without loss of generality assume that the approximation function $\psi$ satisfies $\psi_{\doubleast}(q) := (q^n \log^2(q))^{-1/m} \lesssim \psi(q) \leq \psi_\noverm(q) := q^{-\qdim/\pdim}$ for all sufficiently large $q$. Indeed, for Theorem~\ref{theorem1}, since \eqref{psiseries} converges and $\psi$ is decreasing, we have that
\[
Q^n \psi^m(Q) \asymp \sum_{q\leq Q} q^{n-1} \psi^m(Q) \leq C \df \sum_{q=1}^\infty q^{n-1} \psi^m(q) < \infty
\]
and thus $\psi(q) \lesssim \psi_\noverm(q)$ for all $q$. Conversely, we can replace $\psi$ by $\max(\psi,\psi_{\doubleast})$ without changing whether the series \eqref{psiseries} converges. For Theorem~\ref{theorem2}, by Remark \ref{remarkij} we can take $\psi = \psii$, and then $\psi_{\doubleast}(q) \leq \psi(q) \leq \psi_\noverm(q)$ for all sufficiently large $q$. Finally, for Theorem~\ref{theorem3} we can without loss of generality suppose that $c \leq 1$, and this implies $\psi_{\doubleast} \leq \psi \leq \psi_\noverm$.
\end{remark}

The following standard lemma will be used in the proof of Theorems \ref{theorem2} and \ref{theorem3}, for example see
\cite[Proposition~1]{BDV}.

\begin{lemma}
\label{lemmafullmeasure}
Let $S \subset I_0$ be a Lebesgue measurable set and suppose that for some constant $0<c\leq 1$, for every sufficiently small interval $I \subset I_0$, we have that
\[
|S\cap I| \ge c |I|.
\]
Then $S$ has full measure in $I_0$.
\end{lemma}

We will also use the following version of the divergence Borel-Cantelli lemma, which can be found as Lemma~GDBC in \cite{BHV24} stated for the uniform probability measure on an interval~$I$.

\begin{lemma}
\label{GDBC}
Let $I\subset\mathbb{R}$ be an interval and let $\{S'_Q\}_{Q\in\QQ}$ be a sequence of Lebesgue measurable subsets of $\R$ indexed by a countable set $\QQ$. Suppose that there exist constants $C>0$ and $c>0$ and a sequence of finite subsets $\QQ_k\subset\QQ\cap[k,\infty)$ such that
\begin{equation}\label{eqn04}
\sum_{Q\in\QQ_k}\frac{|S'_Q\cap I|}{|I|} \ge c
\end{equation}
and
\begin{equation} \label{eqn05}
\sum_{\substack{Q_1<Q_2\\ Q_1,Q_2\in\QQ_k}} \frac{|S'_{Q_1}\cap S'_{Q_2} \cap I|}{|I|} \ \le \  C\,  \left(\sum_{Q\in\QQ_k}\frac{|S'_{Q}\cap I|}{|I|}\right)^2
\end{equation}
{for all sufficiently large $k\in\N$.}
Then
$$
\frac{|\limsup_{Q\to\infty} S'_Q\cap I|}{|I|} \ge \frac{1}{2C+c^{-1}}\,.
$$
\end{lemma}

We bring this section to a close by introducing various pieces of specialized notation that will be used in the proofs of our results.

\medskip

\begin{notation}~
\label{notationQRcap}
\begin{itemize}
\item We let $\QQ = \{2^k : k\in\N\}$.
\item We consider a two-parameter family of regions in $\R^{m+n}$ defined as follows: for all $a,b > 0$, we let
\[
\RR(a,b) = \{(\pp,\qq) \in \R^m\times\R^n : \|\pp\| \leq a , \; \|\qq\| \leq b\}\,.
\]
\item We use capital Greek letters to denote functions which are the product of the function $q\mapsto q^{n/m}$ with a function denoted with a lowercase Greek letter: for example $$\Psi(q) := q^{\qdim/\pdim} \psi(q) \, , \quad \Phi(q) := q^{\qdim/\pdim} \phi(q) \, ,  \quad \Theta(q) := q^{\qdim/\pdim} \theta(q)  \, . $$ The exception to this is $\Delta$, which is instead defined by formula \eqref{Deltadef}.
\item  For each $Q \geq 1$ and for each $m\times n$ matrix $\bfalpha$ we let
\begin{align*}
g_Q \df \left[\begin{array}{ll}
Q^{\qdim/\pdim} \mathrm I_\pdim &\\
& Q^{-1} \mathrm I_\qdim
\end{array}\right]\quad\text{and}\quad
u_\bfalpha \df \left[\begin{array}{ll}
\mathrm I_\pdim & \bfalpha\\
& \mathrm I_\qdim
\end{array}\right].
\end{align*}
Here $\mathrm I_d$ denotes the $d$-dimensional identity matrix.
\end{itemize}
\end{notation}

\section{Proof of Theorems \ref{theorem1}--\ref{theorem3}, part 1}
\label{sectionpart1}

Throughout this section, we fix $m\geq 2$, $n\geq 1$, an $m\times n$ matrix $\bfalpha$, a doubling approximation function $\psi$ such that $\psi(q) \lesssim q^{-\qdim/\pdim}$, a curve $\CC \subset \R^\pdim$, and a $C^m$ parameterization $\ff:I_0\to \CC$ with Wronskian bounded from below by a fixed positive constant. Here $I_0 \subset \R$ is an interval. We can also assume without loss of generality that $\ff''$ is uniformly bounded on $I_0$ by a  fixed constant. Recall that in the context of Theorems~\ref{theorem1}--\ref{theorem3} we are interested in the Lebesgue measure of the set
\[
\ff^{-1}\big(\Twist_\psi(\bfalpha)\big) \, ;
\]
that is, the push-forward Lebesgue measure of the set $\CC_{\ff}  \cap  \Twist_\psi(\bfalpha) $.

\bigskip

{\it \underline{Step 1}: Rewriting the problem.}
We begin by defining a new set $\TTwist_\psi(\bfalpha)$ which is ``equivalent up to a constant'' to $\Twist_\psi(\bfalpha)$. For each $Q\geq 1$, let $A_{\psi,Q}$ denote the set of vectors $\bfbeta\in\R^\pdim$ such that for some $(\pp,\qq)\in\Z^\pdim\times\Z^\qdim$
\begin{equation}\label{vb3.2}
\|\bfalpha\qq+\pp+\bfbeta\|\le \psi(Q),\qquad\|\qq\|\le Q\,.
\end{equation}
Finally, let
\[
\TTwist_\psi(\bfalpha) := \limsup_{\QQ\ni Q\to\infty} A_{\psi,Q},
\]
where we recall that $\QQ := \{2^k : k\in\N\}$.

\begin{lemma}
\label{lemmadani}
Let $\bfalpha$ be an $m\times n$ matrix, $\psi$ be a nonincreasing doubling function and let $\mathcal{Z}_\bfalpha=\{\bfalpha\qq+\pp:\qq\in\Z^\qdim,\pp\in\Z^\pdim\}$.
There exists a constant $C \ge 1$ depending on $\psi$ such that
\[
\TTwist_{\psi}(\bfalpha)\setminus\mathcal{Z}_\bfalpha \subset \Twist_\psi(\bfalpha) \subset \TTwist_{C\psi}(\bfalpha).
\]
\end{lemma}
\begin{proof}
Suppose that $\bfbeta\in \Twist_\psi(\bfalpha)$. Then there are infinitely many  $(\pp,\qq)\in\Z^\pdim\times(\Z^\qdim\butnot\{\0\})$ such that $\|\bfalpha\qq + \pp + \bfbeta\| \leq \psi(\|\qq\|)$. Fix any of these $(\pp,\qq)$ and
write $Q/2 < \|\qq\| \leq Q$ for some $Q\in\QQ$, that is, $Q=2^k$ for some $k\in\N$. Since $\psi$ is doubling, we have $\psi(\|\qq\|) \asymp \psi(Q)$ where the implied constant depends only on $\psi$. Therefore
\begin{equation}\label{vb3.3}
\|\bfalpha\qq+\pp+\bfbeta\|\ll\psi(Q),\qquad\|\qq\|\le Q\,.
\end{equation}
Since the above holds for infinitely many $(\pp,\qq)$, \eqref{vb3.3} holds for some $(\pp,\qq)$ in question for infinitely many $Q\in\QQ$ and so we have that $\bfbeta\in \TTwist_{C\psi}(\bfalpha)$ for some $C$ depending on $\psi$ only. The left hand side inclusion is similar. Indeed,  first notice \eqref{vb3.2} already implies \eqref{vb1.2}. Then to ensure that for a given  $\bfbeta\in \TTwist_{\psi}(\bfalpha)\setminus\mathcal{Z}_\bfalpha$ there are infinitely many different $(\pp,\qq)\in\Z^\pdim\times(\Z^\qdim\butnot\{\0\})$,  we use the fact that $\bfalpha\qq+\pp+\bfbeta\neq0$ which follows directly from the fact that $\bfbeta\not\in\mathcal{Z}_\bfalpha$.
\end{proof}

In view of Lemma~\ref{lemmadani} and the fact that $\mathcal{Z}_\bfalpha$ is a countable set, to prove the main theorems it suffices to estimate the measure of the set
\begin{equation}\label{vb03.3}
\ff^{-1}(\TTwist_\psi(\bfalpha)) := \limsup_{Q\to\infty} S_Q   \quad {\rm where}   \quad
S_Q = S_{\psi,Q} = \ff^{-1}(A_{\psi,Q}) \;\;\; (Q\in\QQ)\,.
\end{equation}
In pursuit of this goal we proceed to estimate the measure of the sets $ S_Q $.
A sufficiently accurate estimate on $|S_Q|$ will suffice to directly prove Theorem~\ref{theorem1}, but to prove Theorem~\ref{theorem3} we will need to estimate $|S_Q\cap I|$ where $I \subset I_0$ is any interval, so as to apply Lemma~\ref{lemmafullmeasure}. The proof of Theorem~\ref{theorem2} will be even more involved: we will need to instead consider a family of sets $(\SQnew)_{Q\in\QQ}$ such that $\SQnew \subset S_Q$ for all $Q\in\QQ$, and such that expressions of the form $|\SQnew\cap I|$ and $|S'_{Q_1} \cap S'_{Q_2}\cap I|$ can both be bounded well enough to apply Lemma~\ref{GDBC}.

To proceed further, we observe that for any $\gamma > 0$ and any measurable $S\subset \R$, by Fubini's theorem applied to $\chi_{\{s+\gamma x\in S\}}$, we have that
\begin{equation}
\label{gammaperturbation}
|S| = \frac12 \int \big|\big\{\tmod\in [-1,1] : \tbase + \gamma \tmod \in S\big\}\big| \;\dee \tbase.
\end{equation}
To estimate this integral when $S = S_Q$, it will be helpful to find a condition ``equivalent up to a constant'' to the condition $\tbase + \gamma \tmod \in S_Q$. Recall that $\tbase + \gamma \tmod \in S_Q$ if and only if $\ff(\tbase + \gamma \tmod) \in A_{\psi,Q}$, where $A_{\psi,Q}$ is given by \eqref{vb3.2}. Thus $\tbase + \gamma \tmod \in S_Q$ if and only if for some $(\pp,\qq)\in\Z^\pdim\times\Z^\qdim$
\begin{equation}\label{vb3.50}
\|\bfalpha\qq+\pp+\ff(\tbase + \gamma \tmod)\|\le \psi(Q),\qquad\|\qq\|\le Q\,.
\end{equation}
Naturally, to obtain this ``equivalent up to a constant'' condition we will replace $\ff(\tbase + \gamma \tmod)$ in \eqref{vb3.50} by the linear part of its Taylor expansion. Here we are thinking of $\tbase$ as ``fixed'' as opposed to $\tmod$ which is ``variable'', though we keep in mind that in the end we will have to integrate with respect to $\tbase$ in \eqref{gammaperturbation}. Now fix a function $\theta:\QQ \to \Rp$ satisfying
\begin{equation}\label{vb3.5}
Q^{-\qdim/\pdim} \leq \theta(Q) = o(\psi^{1/2}(Q)),
\end{equation}
to be specified later, with the intention of letting $\gamma = \theta(Q)$ when $S = S_Q$. We fix $\tbase \in I_0$ and $\tmod\in [-1,1]$ such that $\tbase + \theta(Q) \tmod \in I_0$. By Taylor's theorem, we have that
\begin{align}
\ff\big(\tbase + \theta(Q) \tmod\big)
&= \ff(\tbase) + \theta(Q) \tmod \ff'(\tbase) + O\big(\theta^2(Q)\big)\,.\label{vb3.7}
\end{align}
Now for each $c > 0$, define $\TS_{Q,\tbase}(c)$ be the set of all $\tmod\in [-1,1]$ such that for some $(\pp,\qq)\in\Z^\pdim\times\Z^\qdim$
\begin{equation}\label{vb3.2+}
\|\bfalpha\qq+\pp+\ff(\tbase)+\theta(Q) \tmod \ff'(\tbase)\|\le c\psi(Q),\qquad\|\qq\|\le Q\,.
\end{equation}
Note that by \eqref{vb3.5}, it follows that $O\big(\theta^2(Q)\big)=o(\psi(Q))$. Then, on comparing \eqref{vb3.50}${}_{\gamma=\theta(Q)}$ and \eqref{vb3.2+} and using \eqref{vb3.7} we immediately obtain the following statement.

\begin{lemma}
\label{lemmaSQs}
If $Q$ is sufficiently large and the closed ball $B(\tbase,\theta(Q))$ is contained in $I_0$, then
\[
\TS_{Q,\tbase}(\tfrac12) \;\subset\; \{\tmod\in [-1,1] : \tbase + \theta(Q) \tmod \in S_Q\} \;\subset\; \TS_{Q,\tbase}(\tfrac32).
\]
\end{lemma}

\bigskip

{\it \underline{Step 2}: Estimating $|\TS_{Q,\tbase}(c)|$.}
Using the notation introduced in Section~\ref{prelim},
define
\begin{equation}\label{vb3.9}
\hspace*{5ex}\point_{\tbase} = (Q^{\qdim/\pdim} \ff(\tbase),\0)\qquad\text{and}\qquad  \Lambda_Q = g_Q u_\bfalpha \Z^\dimsum\,,
\end{equation}
\begin{align*}
L_{\tbase} = [-\Theta(Q),\Theta(Q)] \cdot (\ff'(\tbase),\0)\qquad\text{and}\qquad
 R_c = \RR(c\Psi(Q),1)\,,\hspace*{5.8ex}
\end{align*}
where $\Theta$ and $\Psi$ are as in Notation \ref{notationQRcap}.
Then, by \eqref{vb3.2+}, for all $\tmod\in [-1,1]$
\begin{equation}
\label{Qapprox2}
x\in \TS_{Q,\tbase}(c)\quad\Longleftrightarrow\quad \point_{\tbase} + \Theta(Q)\tmod (\ff'(\tbase),\0) \in \Lambda_Q + R_c\,.
\end{equation}
The condition on the right hand side  of \eqref{Qapprox2} defines the intersection of the line segment $\point_{\tbase}+L_{\tbase}$ in $\R^{m+n}$ with the collection of rectangles $\Lambda_Q+R_c$ centred at points in the lattice $\Lambda_Q$. Moreover, the map $\tmod\mapsto \point_{\tbase}+\Theta(Q)\tmod(\ff'(\tbase),\0)$ is a linear bijection between the interval $[-1,1]$ and the line segment $\point_{\tbase}+L_{\tbase}$. Consequently,
\begin{equation}\label{vb3.10}
|\TS_{Q,\tbase}(c)| = 2\frac{|(\pointtbase + \Ltbase)\cap (\LambdaQ + R_c)|}{|\Ltbase|}\,,
\end{equation}
where $|\cdot|$ denotes one-dimensional Hausdorff measure (or equivalently the length of a subset of a line).

\begin{lemma}
\label{lemma1}
There exists a constant $C_1 \geq 1$ such that if
\begin{equation}
\label{LRdual}
C_1(L_{\tbase} + R_{1/2})^* \cap \LambdaQ^* = \{\0\}
\end{equation}
then for $c \in \{1/2,3/2\}$ we have that
\begin{equation}
\label{upperbound}
|\TS_{Q,\tbase}(c)| \lesssim \Psi^\pdim(Q)\,.
\end{equation}
If in addition
\begin{equation}
\label{notpsiapprox}
2R_{3/2} \cap \LambdaQ = \{\0\}
\end{equation}
then
\begin{equation}
\label{goodt1}
|\TS_{Q,\tbase}(c)| \asymp \Psi^\pdim(Q).
\end{equation}
The implied constants above do not depend on $Q$. \end{lemma}
\begin{proof}
Suppose that \eqref{LRdual} holds, fix $c \in \{1/2,3/2\}$, and consider the convex, centrally symmetric region $S = \Ltbase+R$, where $R_{1/2} \subset R = R_c \subset R_{3/2}$. Then, by \eqref{LRdual}, we have that$\lambda_1(S^* ; \LambdaQ^*) \geq C_1$ and thus by the duality principle for Minkowski minima \cite[Theorem~VIII.5.VI]{Cassels3}, we have that $\lambda_{\dimsum}(S ; \LambdaQ) \lesssim \tfrac 1{C_1}$, where the implied constant depend on $m$ and $n$ only. Thus $\LambdaQ$ has a fundamental domain $D$ such that $D \subset \tfrac C{C_1} S$, where $C$ is a constant depending on $m$ and $n$ only. Choosing $C_1 \geq 4C$, we have that $D \subset \tfrac14 S$ and hence
\[
\tfrac14 S \subset \bigcup_{\bb\in \tfrac12S\cap (\LambdaQ - \pointtbase)} (\bb + D) \subset \bigcup_{\bb\in S\cap (\LambdaQ - \pointtbase)} (\bb + D) \subset \tfrac 54 S\,.
\]
Since $|D| = \Covol(\LambdaQ) = 1$ and the above unions are disjoint, taking volumes yields
\begin{equation}
\label{asympS}
\#\left(\tfrac12S\cap (\LambdaQ - \pointtbase)\right) \asymp \#\big(S\cap (\LambdaQ - \pointtbase)\big) \asymp |S|.
\end{equation}
Furthermore, it is easy to see that
\begin{align}
\label{vb3.15}\bigcup_{\bb\in \tfrac12S\cap (\LambdaQ - \pointtbase)} (\pointtbase+\Ltbase)&\cap (\pointtbase+\bb+R)
\;\subset\;
(\pointtbase + \Ltbase)\cap (\LambdaQ + R)\\[-2ex]
&\;\subset
\bigcup_{\bb\in S\cap (\LambdaQ - \pointtbase)} (\pointtbase+\Ltbase)\cap (\pointtbase+\bb+R).\nonumber
\end{align}
Hence, if \eqref{notpsiapprox} holds, then the union on the left-hand side is disjoint and thus taking lengths in \eqref{vb3.15}  and applying \eqref{asympS} yields
\begin{equation}\label{vb3.16}
|S| \cdot \min_{\bb \in \tfrac12 S} |\Ltbase\cap (\bb + R)|
\lesssim
|(\pointtbase + \Ltbase)\cap (\LambdaQ + R)|
\lesssim
|S| \cdot \max_{\bb \in S} |\Ltbase\cap (\bb + R)|.
\end{equation}
We note that even if \eqref{notpsiapprox} does not hold, the right hand side of these inequalities still holds. Now from the definitions of $\Ltbase$ and $R$, it is evident that $|\Ltbase\cap (\bb + R)| \lesssim \Psi(Q)$ for all $\bb$. On the other hand, if $\bb\in \tfrac12 S$ then we can write $\bb = \cc + \dd$ where $\cc\in \tfrac12 \Ltbase$ and $\dd\in \tfrac12 R$, and thus
\[
|\Ltbase\cap (\bb + R)| = |(\Ltbase - \cc)\cap (\dd + R)| \geq |\tfrac12 \Ltbase \cap \tfrac12 R| \asymp |\Ltbase\cap R| \asymp \Psi(Q),
\]
where in the last estimate we have used the assumption that $\theta(Q) \geq Q^{-n/m} \gg \psi(Q)$. Combining the estimates for $|\Ltbase\cap (\bb + R)|$ with \eqref{vb3.16} yields
\[
|(\pointtbase + \Ltbase)\cap (\LambdaQ + R)| \asymp |S| \cdot \Psi(Q)\,,
\]
where only the lower bound is conditional on \eqref{notpsiapprox}.
Since $\Psi(Q) \lesssim 1 \leq \Theta(Q)$, up to constant factors $S$ is a box of dimensions $\Theta(Q)\times \Psi(Q) \times\ldots\times\Psi(Q) \times 1\times\ldots\times 1$, and thus
\[
|S| \asymp \Theta(Q) \cdot\Psi^{\pdim-1}(Q)\,.
\]
Plugging this into the previous formula yields
\[
|(\pointtbase + \Ltbase)\cap (\LambdaQ + R)| \asymp \Theta(Q) \Psi^m(Q) \asymp |\Ltbase|\cdot \Psi^m(Q)\,,
\]
where again only the lower bound is conditional on \eqref{notpsiapprox}. Together with \eqref{vb3.10} this completes the proof.
\end{proof}

From now on, a point $\tbase\in I_0$ will be called \emph{good} (with respect to $Q$) if \eqref{LRdual} holds, and \emph{bad} otherwise.

\medskip

{\it \underline{Step 3}: Estimating the probability that $\tbase$ is bad.} Let $B = B_Q$ denote the set of points that are bad with respect to $Q$.
Note that the polar body $(L_{\tbase} + R_{1/2})^*$ involved in defining bad points via \eqref{LRdual} satisfies
\[
(L_{\tbase} + R_{1/2})^* \subset \left\{
(\pp,\qq) \in \R^\dimsum \left|\;
|\pp\cdot \ff'(\tbase)| \leq \frac{1}{\Theta(Q)},
\|\pp\| \leq \frac{2}{\Psi(Q)},
\|\qq\| \leq 1
\right.\right\}.
\]
Let $\tbase\in B$. Then \eqref{LRdual} fails and therefore there exists $(\pp,\qq)\in C_1(L_\tbase\cap R_c)^* \cap \LambdaQ^* \butnot \{\0\}$. That is, there exists $(\pp,\qq)\in\LambdaQ^*\butnot\{\0\}$ such that
$$
|\pp\cdot\ff'(\tbase)| \leq \frac{C_1}{\Theta(Q)},\qquad\|\pp\| \leq \frac{2C_1}{\Psi(Q)}\qquad\text{and}\qquad \|\qq\| \leq C_1\,.
$$
Hence
\begin{equation}\label{B_Q}
|B_Q| \leq \sum_{\substack{(\pp,\qq) \in \LambdaQ^*\butnot\{\0\} \\ \|\pp\| \leq 2C_1/\Psi(Q) \\ \|\qq\| \leq C_1}} \left|\left\{\tbase\in I_0 : |\pp \cdot \ff'(\tbase)| \leq \frac{C_1}{\Theta(Q)}\right\}\right|.
\end{equation}
The following lemma attributed to Pjartli \cite{MR250978}, whose complete proof can also be found in \cite[Lemmas~2]{Beresnevich_Groshev}, will allow us to estimate the terms on the right-hand side of \eqref{B_Q}.

\begin{lemma}[{\cite[Lemmas~2]{Beresnevich_Groshev}}]\label{Pjartli}
Let $\delta,\rho>0$, $I$ be an interval, and $\phi$ be a $C^k$ function on $I$ such that  $|\phi^{(k)}(x)|\ge\rho$ for all $x\in I$. Then $|\{x\in I : |\phi(x)| < \delta\}| \ll (\delta/\rho)^{1/k}$, where the implied constant depends on $k$ only.
\end{lemma}

Applying this lemma to $\phi(\tbase)=\pp \cdot \ff'(\tbase)$ gives the following

\begin{lemma}
\label{lemmanondegenerate}
For all $\pp\in \R^\pdim$ and $\delta > 0$ we have that
\begin{equation}
\label{nondegenerate}
|\{s\in I_0 : |\pp\cdot\ff'(s)| \leq \delta\}| \lesssim (\delta/\|\pp\|)^{1/(\pdim - 1)}\,,
\end{equation}
where the implied constant depends on $\ff$ and $m$ only.
\end{lemma}

\begin{proof}
If $\pp=\bf0$, \eqref{nondegenerate} is trivially true since the right hand side is $+\infty$, so we can assume that $\pp\neq\bf0$. Since the Wronskian of $\ff'$ is bounded away from $0$ and $\ff^{(k)}$ is bounded on $I_0$ for every $1\le k\le m$, we have that $\max_{1\le k\le m}|\pp \cdot \ff^{(k)}(\tbase)|\gg \|\pp\|$ at each $\tbase\in I_0$. By continuity and replacing $I_0$ with a smaller interval if necessary, we can assume without loss of generality that the above maximum is attained on the same $k$ for each $\tbase\in I_0$. If $k=1$ then \eqref{nondegenerate} becomes trivial, otherwise \eqref{nondegenerate} follows from Lemma~\ref{Pjartli}.
\end{proof}

Combining Lemma~\ref{lemmanondegenerate} and \eqref{B_Q} gives the following estimate
\begin{equation}
\label{estimatesum}
|B_Q| \lesssim \sum_{\substack{(\pp,\qq) \in \LambdaQ^*\butnot\{\0\} \\ \|\pp\| \leq 2C_1/\Psi(Q) \\ \|\qq\| \leq C_1}} \left(\frac{1}{\Theta(Q) \|\pp\|}\right)^{1/(\pdim-1)}.
\end{equation}

Now we obtain a non-trivial bound on the sum in \eqref{estimatesum} under an additional assumption on $\LambdaQ^*$, namely \eqref{assump1} below. Note that, in general, within \eqref{estimatesum} there is no guarantee that $\pp\neq\bf0$ for all indices $(\pp,\qq)$ and so this estimate may be trivial.

\begin{lemma}
\label{lemmasummary}
Let $\bfalpha$, $\psi$, and $\ff:I_0\to \CC$ be as at the start of this section, and let $\theta,\Delta:\QQ\to \Rp$ be functions such that $Q^{-\qdim/\pdim} \leq \theta(Q) = o(\psi^{1/2}(Q))$ and $\Delta(Q) \leq 1$. Then, for all sufficiently large $Q\in\QQ$ such that
\begin{equation}\label{assump1}
\Lambda_Q^*\cap 2R'_Q = \{\0\}\,,\qquad\text{where}\qquad
R' = R'_Q = \RR(\Delta(Q),C_1)\,,
\end{equation}
we have that
\begin{equation}
\label{BQbound}
|B_Q| \lesssim \frac{1}{\Delta^\pdim(Q)} \left(\frac{\Psi(Q)}{\Theta(Q)}\right)^{1/(\pdim-1)} \frac1{\Psi^\pdim(Q)}\,.
\end{equation}
\end{lemma}

\begin{proof}
Let $\Gamma = \tfrac{\Delta(Q)}{\pdim} \Z^\pdim\times\{\0\}$. Fix $(\pp,\qq)\in\LambdaQ^*\butnot\{\0\}$ such that $\|\pp\| \leq 2C_1/\Psi(Q)$ and $\|\qq\| \leq C_1$, that is, $(\pp,\qq)$ is the index of a term in \eqref{estimatesum}. Since $\tfrac{\Delta(Q)}{\pdim} \Z^\pdim$ is $\Delta(Q)$-dense in $\R^m$, there exists $(\pp',\0)\in \Gamma$  such that $(\pp,\qq)\in (\pp',\0) + R'$. On the other hand, since $\LambdaQ^*\cap 2R' = \{\0\}$, the map $(\pp,\qq)\mapsto (\pp',\0)$ is injective, and since $\0\mapsto\0$ it follows that $\LambdaQ^*\butnot\{\0\} \mapsto \Gamma\butnot\{\0\}$. Also, by \eqref{assump1}, we necessarily have that $\|\pp\|\ge \Delta(Q)$. Then, since $\Psi(Q),\Delta(Q) \lesssim 1$, by the triangle inequality, we have that
\begin{align*}
\|\pp'\| &\leq 2C_1/\Psi(Q) + \Delta(Q) \leq C_2/\Psi(Q),\\[0ex]
\|\pp\| &\geq \max\big(\Delta(Q),\|\pp'\| - \Delta(Q)\big) \geq \tfrac12 \|\pp'\|.
\end{align*}
Now, by \eqref{estimatesum}, we obtain that
\begin{align*}
|B_Q| &\lesssim \sum_{\substack{(\pp',\0) \in \Gamma\butnot\{\0\} \\ \|\pp'\| \leq C_2/\Psi(Q)}} \left(\frac{1}{\Theta(Q) \|\pp'\|}\right)^{1/(\pdim-1)}\\[1ex]
&\asymp \frac{1}{\Delta^\pdim(Q)} \left(\frac{1}{\Theta(Q)}\right)^{1/(\pdim-1)} \int_{\Delta(Q) \leq \|\pp\| \leq C_2/\Psi(Q)} \frac{1}{\|\pp\|^{1/(\pdim-1)}} \;\dee\pp\\[1ex]
&\asymp \frac{1}{\Delta^\pdim(Q)} \left(\frac{1}{\Theta(Q)}\right)^{1/(\pdim-1)}
\int_{\Delta(Q)}^{C_2/\Psi(Q)} \frac{1}{p^{1/(\pdim-1)}} p^{\pdim-1} \;\dee p\\[1ex]
&\asymp \frac{1}{\Delta^\pdim(Q)} \left(\frac{\Psi(Q)}{\Theta(Q)}\right)^{1/(\pdim-1)}
\frac1{\Psi^\pdim(Q)}.
\end{align*}
\end{proof}

In order to apply the above findings we now state and prove several lemmas that will be used to verify conditions \eqref{assump1} and \eqref{notpsiapprox}.

\begin{lemma}
\label{lemmaRprime1}
Let $\delta$ be an approximation function, and suppose that the transpose $\bfalpha^T$ is not $\delta$-approximable. Let
\begin{equation}
\label{Deltadef}
\Delta(Q) = \frac12 Q^{-\qdim/\pdim} \delta^{-1}\left(\frac{2C_1}{Q}\right).
\end{equation}
Then for all $Q\geq 1$ sufficiently large, \eqref{assump1} holds.
\end{lemma}

\begin{proof}
By contradiction, suppose that $\Lambda_Q^* \cap 2R'_Q \neq \{\0\}$ for arbitrarily large $Q$. Fix such a $Q$. Since $\Lambda_Q^* = ((g_Q u_\bfalpha)^T)^{-1} \Z^{m+n}$, it follows that there exists $(\pp,\qq)\in\Z^\dimsum\butnot\{\0\}$ such that $((g_Q u_\bfalpha)^T)^{-1}(\pp,\qq) \in 2 R'_Q$. Now
\[
((g_Q u_\bfalpha)^T)^{-1}(\pp,\qq) = \big(Q^{-\qdim/\pdim} \pp, Q (\qq - \bfalpha^T\cdot\pp)\big)
\]
and thus
\[
Q^{-\qdim/\pdim} \|\pp\| \leq 2\Delta(Q), \;\;\;\; Q \|\qq - \bfalpha^T\cdot\pp\| \leq 2C_1.
\]
Rearranging this and applying \eqref{Deltadef} gives
\begin{equation}\label{vb3.24}
\|\qq - \bfalpha^T \cdot \pp\| \leq 2C_1/Q \leq \delta(\|\pp\|)\,.
\end{equation}
Since this has a solution $(\pp,\qq)\in\Z^\dimsum\butnot\{\0\}$ for arbitrarily large $Q$, either $\qq - \bfalpha^T \cdot \pp=0$, or there exist infinitely many such pairs $(\pp,\qq)$ satisfying \eqref{vb3.24}. In either case $\bfalpha^T$ is $\delta$-approximable, contrary to the conditions on the lemma.
\end{proof}

The following lemma is similar, and serves the purpose of verifying \eqref{notpsiapprox}, which is required to use the full power of Lemma~\ref{lemma1}.

\begin{lemma}
\label{lemmaR}
Let $\phi$ be an approximation function, and suppose that $\bfalpha$ is not $\phi$-approximable. Then for all $Q\geq 1$ sufficiently large, we have that
$$\Lambda_Q\cap R''_Q = \{\0\}\,,\qquad\text{where}\qquad
R''_Q = \RR(2^{-n/m} \Phi(2Q),2).
$$
In particular, if $2^{-n/m} \Phi(2Q) \geq 3\Psi(Q)$ then \eqref{notpsiapprox} holds, that is, $\Lambda_Q\cap 2 R_{3/2} = \{\0\}$, since $2 R_{3/2} = \RR(3\Psi(Q),2)$.
\end{lemma}
Note that the condition  $2^{-n/m} \Phi(2Q) \geq 3\Psi(Q)$ is satisfied for all $Q$ sufficiently large if $\psi = \psii$ and $\phi = \phi_i$ are as in \eqref{phidef}.
\begin{proof}
By contradiction, suppose that $\Lambda_Q \cap R''_Q \neq \{\0\}$ for arbitrarily large $Q$. Fix such a $Q$. Then there exists $(\pp,\qq)\in\Z^\dimsum\butnot\{\0\}$ such that $g_Q u_\bfalpha(\pp,\qq) \in R''_Q$. Now
\[
g_Q u_\bfalpha(\pp,\qq) = \big(Q^{\qdim/\pdim}(\bfalpha\cdot\qq + \pp), Q^{-1} \qq\big)
\]
and thus
\[
Q^{\qdim/\pdim} \|\bfalpha\cdot\qq + \pp\| \leq 2^{-n/m} \Phi(2Q), \;\;\;\; Q^{-1} \|\qq\| \leq 2.
\]
Similarly to the proof of Lemma~\ref{lemmaRprime1}, rearranging this gives
\[
\|\bfalpha\cdot\qq + \pp\| \leq \phi(2Q) \leq \phi(\|\qq\|)
\]
for infinitely many $(\pp,\qq)$, thus implying that $\bfalpha$ is $\phi$-approximable, a contradiction.
\end{proof}

Finally, in the proof of Theorem~\ref{theorem3}, we will need the following version of Lemma~\ref{lemmaRprime1} that utilizes a different Diophantine condition on $\bfalpha$.

\begin{lemma}
\label{lemmanonsingular}
Suppose that $\bfalpha$ is nonsingular. Then there exists $\epsilon > 0$ such that if we let $\Delta(Q) = \epsilon$ for all $Q$, then for infinitely many $Q \in \QQ$, \eqref{assump1} holds.
\end{lemma}

\begin{proof}
Since $\bfalpha$ is nonsingular, there exist $\epsilon' > 0$ and arbitrarily large $Q$ such that the system of inequalities
\[
\|\qq\| \leq Q, \;\;\;\; \|\bfalpha\cdot\qq + \pp\| \leq \epsilon' Q^{-\qdim/\pdim}
\]
has no non-trivial integer solution $(\pp,\qq)\in \Z^\dimsum\butnot\{\0\}$. It follows that there exist arbitrarily large $Q\in\QQ$ such that the system of inequalities
\[
\|\qq\| \leq 2 C_1 Q, \;\;\;\; \|\bfalpha\cdot\qq + \pp\| \leq \epsilon' (4 C_1 Q)^{-\qdim/\pdim}
\]
has no nontrivial integer solution $(\pp,\qq)\in \Z^\dimsum\butnot\{\0\}$. Equivalently, if $\Delta(Q) = \epsilon := \frac12 (4C_1)^{-\qdim/\pdim} \epsilon'$, then $\Lambda_Q^* \cap 2 R'_Q = \{\0\}$.
\end{proof}

\section{Proof of Theorems \ref{theorem1}--\ref{theorem3}, part 2}

At this point the proofs of Theorems~\ref{theorem1}, \ref{theorem2}, and~\ref{theorem3} diverge. In the proof of Theorem~\ref{theorem1} we will need to apply Lemma~\ref{lemmasummary} with a large value of $\theta$, whereas in the proofs of Theorems \ref{theorem2} and~\ref{theorem3} we will need to use a relatively small value of $\theta$. We keep the same notation as in Section~\ref{sectionpart1}. We will also assume without loss of generality that $\psi(q) \geq \psi_{\doubleast}(q) := (q^n \log^2(q))^{-1/m}$ for all sufficiently large $q$, or equivalently, by definition,  that $\Psi(q) \geq \log^{-2/\pdim}(q)$.

\subsection{Proof of Theorem~\ref{theorem1}}

In view of Remarks \ref{remark1} and \ref{remark2} and Lemma~\ref{lemmadani}, proving the following  claim will complete the proof of Theorem~\ref{theorem1}.

\begin{claim}\label{claim1}
If \eqref{psiseries} converges and \eqref{NVWA} holds, then $|\ff^{-1}(\TTwist_\psi(\bfalpha))| = 0$.
\end{claim}

\begin{proof}
To begin with we note that, in view of the definitions of $\QQ$ and $\Psi$ (see Section~\ref{prelim}), by Cauchy's condensation test, the convergence of \eqref{psiseries} implies that
\begin{equation}\label{Psiconverge}
\sum_{Q\in\QQ}\Psi^\pdim(Q)<\infty\,.
\end{equation}
In view of \eqref{vb03.3}, by the Borel-Cantelli Lemma, Claim~\ref{claim1} will follow on showing that
\begin{equation}\label{vb}
\sum_{Q\in\QQ}|S_Q|<\infty\,.
\end{equation}
By \eqref{NVWA}, there exists $0 < \epsilon < \qdim/2$ such that
\[
\gamma \df \left(\tfrac\qdim\pdim - \tfrac{\qdim - 2\epsilon}{2\pdim^2 (\pdim-1)}\right)^{-1} > \omega(\bfalpha^T).
\]
Let $\theta(Q): = Q^{-(\qdim+\epsilon)/2\pdim}$, so that $Q^{-\qdim/\pdim} \leq \theta(Q) = o(\psi^{1/2}(Q))$. Let $\delta(q): = q^{-\gamma}$, and let $\Delta$ be as in \eqref{Deltadef}. Then we have that $\Theta(Q) = Q^{(\qdim-\epsilon)/2\pdim}$ (see Notation~\ref{notationQRcap}) and $\Delta(Q) \asymp Q^{1/\gamma - \qdim/\pdim} = Q^{-(\qdim-2\epsilon)/2\pdim^2 (\pdim-1)}$. In particular, we have that $\Delta(Q) \leq 1$ for all sufficiently large $Q$. Since $\gamma > \omega(\bfalpha^T)$, $\bfalpha^T$ is not $\delta$-approximable. Thus, by Lemmas~\ref{lemmaRprime1} and \ref{lemmasummary}, \eqref{BQbound} holds. Thus, since $\Psi(Q) \gg \log^{-2/m}(Q)$, for all sufficiently large $Q\in\QQ$ we have that
\begin{align}
\nonumber|B_Q| &\lesssim \frac{1}{\Delta^\pdim(Q)} \left(\frac{1}{\Theta(Q)}\right)^{1/(\pdim-1)} \frac{1}{(\log^{-2/\pdim}(Q))^\pdim}\\[2ex]
&\asymp Q^{(\qdim-2\epsilon)/2\pdim(\pdim-1)} Q^{-(\qdim-\epsilon)/2\pdim(\pdim-1)} \log^2(Q) = Q^{-\epsilon/2\pdim(\pdim-1)} \log^2(Q)\,.\label{B_Qest2}
\end{align}
By \eqref{gammaperturbation},
\begin{equation}\label{vb4.2}
|S_Q|\le|B_Q|+
\frac12 \int_{\R\butnot B_Q} \big|\big\{\tmod\in [-1,1] : \tbase + \theta(Q) \tmod \in S_Q\big\}\big| \;\dee \tbase.
\end{equation}
Since $S_Q\subset I_0$, the set on the right-hand side of \eqref{vb4.2} is empty whenever $s\not\in I_1:=I_0+[-1,1]$ - a bounded interval obtained by extending $I_0$ by $1$ left and right. Hence, \eqref{vb4.2} implies that
\begin{equation}\label{vb4.6}
|S_Q|\le|B_Q|+
\frac12 \int_{I_1\butnot B_Q} \big|\big\{\tmod\in [-1,1] : \tbase + \theta(Q) \tmod \in S_Q\big\}\big| \;\dee \tbase.
\end{equation}
Now, using Lemmas~\ref{lemmaSQs} and \ref{lemma1} together with \eqref{B_Qest2} gives that
\begin{equation}\label{vb4.3}
|S_Q|\ll Q^{-\epsilon/2\pdim(\pdim-1)} \log^2(Q)+ \Psi^\pdim(Q)
\end{equation}
and since $\QQ$ is a geometric progression, by \eqref{Psiconverge}, we conclude \eqref{vb} and complete the proof.

\end{proof}

\subsection{Proof of Theorem~\ref{theorem2}}
Since in the context of Theorem~\ref{theorem2} the sum \eqref{psiseries} diverges, by Cauchy's condensation test, we have that
\begin{equation}\label{Psidiverge}
\sum_{Q\in\QQ}\Psi^\pdim(Q)=\infty\,.
\end{equation}
Further, in view of Remarks \ref{remark1} and \ref{remark2} and Lemma~\ref{lemmadani}, proving the following claim will complete the proof of Theorem~\ref{theorem2}.

\begin{claim}
\label{claim2}
Suppose that $\psi = \psii$ and $\phi = \phi_i$ are as in \eqref{phidef}, and that $\bfalpha$ is not $\phi$-approximable. Let $C>0$ be a constant. Then $|\ff^{-1}(\TTwist_{C\psi}(\alpha))| = |I_0|$.
\end{claim}

We begin  by  proving   the following  auxiliary statement.

\begin{lemma}
Let $d = \pdim + \qdim$, $\beta > (d-1)m/n$, $\phi$ be as in Claim~\ref{claim2} and \[
\delta(Q) := Q^{-m/n} \Phi^\beta(Q)\,.
\]
Suppose that  $\bfalpha$ is not $\phi$-approximable. Then $\bfalpha^T$ is not $\delta$-approximable.
\end{lemma}

\begin{proof}
Observe that $\Phi(q) \geq \Psi(q) \geq \log^{-2}(q)$ and therefore $g(x) := -\log\Phi(\exp\exp(x)) \leq 2x$. Consequently, since $g$ is a Hardy $L$-function, by \cite[Theorem on p.50]{Hardy}, we have that $0\leq g'(x) \lesssim \frac\del{\del x}(2x) = 2$. It follows that  $g$ is eventually $C$-Lipschitz for some constant $C$, i.e. $|g(y)-g(x)| \leq C|y - x|$ for all sufficiently large $x,y$. Letting $x = \log\log(q)$ and $y = \log\log(q')$, exponentiating both sides, and applying the definition of $g$ yields that
\begin{equation}
\label{Philip}
\Phi(q) \asymp \Phi(q') \quad \text{if}   \quad  \log(q) \asymp \log(q')
\end{equation}
(with the implied constant on the left-hand asymptotic depending on the implied constant on the right-hand asymptotic). Now,  aiming for a  contradiction, suppose that $\bfalpha^T$ is $\delta$-approximable. Then by Khintchine's transference principle \cite[Theorem V.II]{Cassels}, $\bfalpha$ is $\gamma$-approximable, where $\gamma$ is the unique function such that
\[
\gamma\left((d-1)q^{m/(d-1)}\delta^{(1-m)/(d-1)}(q)\right) = (d-1) q^{(1-n)/(d-1)} \delta^{n/(d-1)}(q) \text{ for all } q.
\]
Since $\bfalpha$ is not $\phi$-approximable, it follows that $\gamma(Q) > \phi(Q)$ for arbitrarily large $Q$. Hence for arbitrarily large $q$ we have
that
\begin{equation}
\label{khinchintransference}
\phi\left((d-1)q^{m/(d-1)}\delta^{(1-m)/(d-1)}(q)\right) < (d-1) q^{(1-n)/(d-1)} \delta^{n/(d-1)}(q).
\end{equation}
Now by the definition of $\delta$,
\begin{align*}
q^{m/(d-1)}\delta^{(1-m)/(d-1)}(q)
&= q^{m/n} \Phi^{\beta (1-m)/(d-1)}(q),\\
q^{(1-n)/(d-1)} \delta^{n/(d-1)}(q)
&= q^{-1} \Phi^{\beta n/(d-1)}(q)
\end{align*}
and on the other hand, $\phi(Q) = Q^{-n/m} \Phi(Q)$. Thus, \eqref{khinchintransference} becomes
\[
\big(q^{m/n} \Phi^{\beta (1-m)/(d-1)}(q)\big)^{-n/m} \Phi\big(q^{m/n} \Phi^{\beta (1-m)/(d-1)}(q)\big) \lesssim q^{-1} \Phi^{\beta n/(d-1)}(q),
\]
and rearranging gives
\[
\Phi\big(q^{m/n} \Phi^{\beta (1-m)/(d-1)}(q)\big) \lesssim \Phi^{\beta n/m(d-1)}(q).
\]
Applying \eqref{Philip} results in
\[
\Phi(q) \lesssim \Phi^{\beta n/m(d-1)}(q)\,,
\]
and since $\Phi(q) \to 0$, this gives a contradiction for $\beta > m(d-1)/n$.
\end{proof}

\bigskip

\begin{proof}[Proof of Claim~\ref{claim2}]
Since
$\delta^{-1}(1/Q) \asymp Q^{n/m} \Phi^{\beta n/m}(Q)$,
we have that
\begin{equation}
\label{Deltabound}
\Delta(Q) \asymp \Phi^{\beta n/m}(Q),
\end{equation}
where $\Delta$ is as in \eqref{Deltadef}. In particular, $\Delta(Q) \leq 1$ for all sufficiently large $Q$.

Let
\begin{align}
\label{thetadef1}
\eta(Q) &:= 1/\Psi(Q),\\
\label{thetadef2}\Theta(Q) &:= \eta(Q) \Psi^{1-\pdim(\pdim-1)}(Q) \Delta^{-\pdim(\pdim-1)}(Q)\,,\\
\label{thetadef3}\theta(Q) &:= Q^{-\qdim/\pdim}\Theta(Q).
\end{align}
(We will later in another context apply \eqref{thetadef2} and \eqref{thetadef3} with a different value of $\eta$.) Since $\Psi(Q)\to 0$ we have that $\eta(Q)\to \infty$. Since $\Delta(Q) \leq 1$ for all sufficiently large $Q$, we have $\Theta(Q) \geq 1$ for all sufficiently large $Q$. Note that
\[
\frac{\theta(Q)}{\psi^{1/2}(Q)} = \frac{Q^{-\qdim/\pdim}}{Q^{-\qdim/2\pdim}} \frac{\eta(Q) \Psi^{1-\pdim(\pdim-1)}(Q) \Delta^{-\pdim(\pdim-1)}(Q)}{\Psi^{1/2}(Q)} = o(1).
\]
Thus, by Lemmas \ref{lemmaRprime1} and \ref{lemmasummary}, for all but finitely many $Q\in\QQ$ we have that
\begin{equation}
\label{BQo1}
|B_Q| \lesssim \frac{1}{\Delta^\pdim(Q)} \left(\frac{\Psi(Q)}{\Theta(Q)}\right)^{1/(\pdim-1)}
\frac1{\Psi^\pdim(Q)} = \eta^{-1/(\pdim-1)}(Q) = o(1).
\end{equation}
Write $I_0 = [a,b]$, and let $A_Q$ be a maximal $3\theta(Q)$-separated subset of $[a+\theta(Q),b-\theta(Q)]\butnot B_Q$. For each $\tbase \in A_Q$, let
\[
E(\tbase) := \tbase + \theta(Q) \TS_{Q,\tbase}(1/2),
\]
where $\TS_{Q,\tbase}(1/2)$ is defined by \eqref{Qapprox2}. By Lemma~\ref{lemmaSQs}, we have that $E(\tbase) \subset S_{\psi,Q}=S_Q$ for all but finitely many $Q\in\QQ$. Next, let
\begin{align}
\label{vb4.18}\SQorig &:= \bigcup_{\tbase\in A_Q} E(\tbase) \subset S_{\psi,Q},&
\SQnew &:= \NN(\SQorig,2C_2\psi(Q)) \subset S_{(1+2C_2^2) \psi,Q}\,,
\end{align}
where $\NN(S,\epsilon)$  denotes an $\epsilon$ neighbourhood of a set $S$, and $C_2 \geq 1$ is the bi-Lipschitz constant of $\ff$. We claim the following:

\begin{subclaim}\label{subclaim1}
For any interval $I\subset I_0$ and any $k\in\N$ there exist $Q^-\ge k$ such that
\begin{equation}
\label{SQprimelowerbound}
\frac{|\SQnew\cap I|}{|I|} \asymp \Psi^\pdim(Q) \text{ for all } \QQ\ni Q \geq Q^-
\end{equation}
and for infinitely many $Q^+\in\QQ$ with $Q^+>Q^-$ we have that
\begin{equation}
\label{quasiindep}
\sum_{Q_1,Q_2 \in \QQ'} \frac{|S'_{Q_1} \cap S'_{Q_2}\cap I|}{|I|} \lesssim \sum_{Q\in\QQ'} \Psi^\pdim(Q) + \left(\sum_{Q\in\QQ'} \Psi^\pdim(Q)\right)^2,
\end{equation}
where $\QQ' := \QQ\cap [Q^-,Q^+]$.
\end{subclaim}

While postponing the proof of Subclaim~\ref{subclaim1} until the next subsection, we now finish the proof of Claim~\ref{claim2}. By Subclaim~\ref{subclaim1}, \eqref{Psidiverge}, and the fact that $\Psi(Q)\le1$ for all sufficiently large $Q$, choose as we may a sequence $\QQ_k:=\QQ\cap[Q^-,Q^+]\subset [k,\infty)$ such that
$$
\sum_{Q\in \QQ_k}\Psi(Q)^m\asymp 1
$$
and  \eqref{SQprimelowerbound} and \eqref{quasiindep} hold. Then, inequalities \eqref{SQprimelowerbound} and \eqref{quasiindep}  verify \eqref{eqn04} and \eqref{eqn05}, making Lemma~\ref{GDBC} applicable. Then, in view of \eqref{vb03.3}, by Lemma~\ref{GDBC}, we have that
\begin{equation}\label{vb4.17}
|\ff^{-1}(\TTwist_{C\psi}(\bfalpha))\cap I| \gtrsim |I|\,.
\end{equation}
Using Lemma~\ref{lemmafullmeasure} completes the proof of Claim~\ref{claim2}.
\end{proof}

\medskip

\subsection{Proof of Subclaim~\ref{subclaim1}}

To begin with we write
$$
\SQnew = \bigcup_{J\in\II_Q} J\,,
$$
where $\II_Q$ is a collection of disjoint intervals. Clearly, by \eqref{vb4.18}, we have that $|J| \geq 4C_2\psi(Q)$ for all $J\in \II_Q$. We will prove the following two subclaims regarding the size and separation of the intervals in $\II_Q$. In what follows, $d(J_1,J_2)$ stands for the (infimal) distance between intervals $J_1$ and $J_2$.

\begin{subclaim}[Separation]
\label{subclaimseparation}
For distinct $J_1,J_2 \in \II_Q$, we have that $\dist(J_1,J_2) \gtrsim \phi(Q)$.
\end{subclaim}
\begin{subclaim}[Size]
\label{subclaimsize}
$4 C_2\psi(Q) \leq |J| \leq 6C_2\psi(Q)$ for all $J\in \II_Q$.
\end{subclaim}

\begin{proof}[Proof of Subclaims~\ref{subclaimseparation} and \ref{subclaimsize}]
Fix $J_1,J_2\in \II_Q$ (not necessarily distinct) and for each $i = 1,2$ fix a point $t_i \in \SQorig \cap J_i \subset S_{\psi,Q} \cap J_i$, where $S''_Q$ is as in \eqref{vb4.18}. Then $\ff(t_i) \in A_{\psi,Q}$ and thus there exists $\rr_i\in \Lambda_Q$ such that
\[
(Q^{\qdim/\pdim}\ff(t_i),\0) - \rr_i \in \RR(\Psi(Q),1).
\]
If $\rr_1 = \rr_2$, then we have that
\[
(Q^{\qdim/\pdim}\ff(t_2) - Q^{\qdim/\pdim}\ff(t_1),\0) \in \RR(2\Psi(Q),2)
\]
and thus
$\|\ff(t_2) - \ff(t_1)\| \leq 2\psi(Q)$.
Since $C_2$ is the bi-Lipschitz constant of $\ff$, it follows that
\begin{equation}
\label{CpsiQ}
|t_2 - t_1| \leq 2C_2 \, \psi(Q),
\end{equation}
and thus $[t_1,t_2] \subset \SQnew$ (or $[t_2,t_1] \subset \SQnew$). Consequently, $J_1 = J_2$. So $\rr_1 = \rr_2$ implies $J_1 = J_2$.

On the other hand, suppose that $\rr_1 \neq \rr_2$. Since
\[
\rr_2 - \rr_1 \in \RR\left(2\Psi(Q) + C_2 Q^{\qdim/\pdim}|t_2 - t_1|,2\right),
\]
by Lemma~\ref{lemmaR} we have that
\[
2^{-n/m} \Phi(2Q) \leq 2\Psi(Q) + C_2 Q^{\qdim/\pdim}|t_2 - t_1|\,.
\]
Since $\Psi(Q) \leq (1/4) 2^{-n/m} \Phi(2Q)$ for all $Q$ sufficiently large (cf. Remark \ref{remarkij}), we have that
\[
\phi(2Q) = O(|t_2 - t_1|)
\]
and thus $|t_2 - t_1| \geq c\phi(Q)$ for some constant $c > 0$. Since $t_1,t_2$ were arbitrary points in $\SQorig\cap J_1,\SQorig\cap J_2$  respectively, it follows that
\begin{itemize}
\item[(A)] If $J_1 \neq J_2$, then $\dist(J_1,J_2) \geq c\phi(Q) - 4C_2\psi(Q) \geq c\phi(Q)/2$ for all $Q$ sufficiently large (cf. Remark \ref{remarkij}). This completes the proof of Subclaim \ref{subclaimseparation}.   \smallskip
\item[(B)] If $J_1 = J_2$, then the above calculation shows that the case $\rr_1 \neq \rr_2$ is impossible as it would result in a gap in $J_1$ of size $\geq c\phi(Q)/2 > 0$ contradicting that $J_1$ is an interval. It follows that for all $t_i \in \SQorig\cap J_i$ we have that  \eqref{CpsiQ} holds, and thus  $|J_i| \leq 6C_2\psi(Q)$. On the other hand, it is clear from the definitions of $\SQnew$ and $J_i$ that $|J_i| \geq 4 C_2 \psi(Q)$; this completes the proof of Subclaim \ref{subclaimsize}.
\qedhere\end{itemize}
\end{proof}

The next two subclaims verify condition \eqref{SQprimelowerbound}.

\begin{subclaim}[Lower bound]
\label{lowerbound}
For any interval $I\subset I_0$ and sufficiently large $Q\in\QQ$ we have that $|\SQnew\cap I|\gg \Psi^m(Q)\cdot|I|$.
\end{subclaim}

\begin{proof}[Proof of Subclaim~\ref{lowerbound}]
Given $I \subset I_0$, and  $Q\in\QQ$, let $J_Q = \{t\in I : B(t,\theta(Q)) \subset I\}$ and $J_Q' = \{t\in I : B(t,4\theta(Q)) \subset I\}$. Then for all but finitely many $Q\in\QQ$ (depending on $I$),
\begin{align*}
|\SQnew\cap I|
~&\geq~ |\SQorig\cap I|~\geq \sum_{\tbase\in A_Q\cap J_Q} |\tbase + \theta(Q) \TS_{Q,\tbase}(1/2)|\\[0ex]
&\asymp~ \#(A_Q\cap J_Q) \theta(Q) \Psi^\pdim(Q) \hspace*{12.5ex}\text{(by Lemmas \ref{lemma1} and \ref{lemmaR})}\\[1ex]
&\gtrsim~ |J_Q'\butnot B_Q| \cdot \Psi^\pdim(Q) \hspace*{5ex}\text{(since $J_Q' \butnot B_Q \subset \NN(A_Q\cap J_Q,3\theta(Q))$)}\\[1ex]
&\geq~ (|I| - |B_Q| - 8\theta(Q)) \cdot \Psi^\pdim(Q)\\[1ex]
&\geq~ \Psi^\pdim(Q) \cdot |I|/2. \hspace*{37ex}\text{(by \eqref{BQo1})}
\end{align*}
\end{proof}

\begin{subclaim}[Upper bound]
\label{subclaimdistribution}
On any interval $I \subset I_0$, we have that
\[
|\SQnew\cap I| \lesssim \Psi^\pdim(Q)\cdot (|I|+\theta(Q)).
\]
In particular, if $|I| \gtrsim \theta(Q)$ then  $|\SQnew\cap I| \lesssim \Psi^\pdim(Q)\cdot |I|$.
\end{subclaim}

\noindent Note that by taking $Q_-$ sufficiently large, we get the upper bound in \eqref{SQprimelowerbound}.

\begin{proof}[Proof of Subclaim~\ref{subclaimdistribution}]
Suppose $Q$ is large enough so that $\Theta(Q) \geq 2 C_2^2$ and $\Psi(Q) \leq 1$, and note that this implies $\theta(Q) \geq 2 C_2 \psi(Q)$. Then
\begin{align*}
|\SQnew&\cap I|
\leq \sum_{\tbase\in A_Q\cap \NN(I,\theta(Q))} |\NN(E(\tbase),2C_2\psi(Q))|\\
&= \theta(Q) \sum_{\tbase\in A_Q\cap \NN(I,\theta(Q))} |\NN(\TS_{Q,\tbase}(1/2),2C_2 \psi(Q)/\theta(Q))| \noreason\\
&\leq \theta(Q) \sum_{\tbase\in A_Q\cap \NN(I,\theta(Q))} |\TS_{Q,\tbase}(3/2)|\hspace*{5ex}\text{(since $2C_2\psi(Q)/\theta(Q) \leq C_2^{-1} \Psi(Q)$)}\\
&\asymp \theta(Q) \sum_{\tbase\in A_Q\cap \NN(I,\theta(Q))} \Psi^\pdim(Q) \hspace*{18.5ex}\text{(by Lemmas \ref{lemma1} and \ref{lemmaR})}\\[2ex]
&\leq \Psi^\pdim(Q) \cdot |\NN(I,2\theta(Q))|\hspace*{18ex}\text{(since $A_Q$ is $3\theta(Q)$-separated)}\\[2ex]
&\asymp \Psi^\pdim(Q) \cdot (|I|+\theta(Q)).
&&\qedhere\end{align*}
\end{proof}

It remains to prove the following ``independence'' statment.

\begin{subclaim}[Quasi-independence estimate]\label{QIE}
Given any interval $I\subset I_0$, for all sufficiently $Q^-$ and any $Q^+$, sufficiently large in terms of $Q^-$,  estimate \eqref{quasiindep} holds.
\end{subclaim}

\begin{proof}[Proof of Subclaim~\ref{QIE}]
Fix $I \subset I_0$ and $Q_1,Q_2\in\QQ$ such that $Q_1 \leq Q_2$, both large in terms of $I$. Then, we trivially have that
\begin{align}
\nonumber\frac{|S'_{Q_1}\cap S'_{Q_2} \cap I|}{|I|} & \lesssim \frac{|S'_{Q_1} \cap \NN(I,6C_2 \psi(Q_1))|}{|I|}\cdot \max_{J\in \II_{Q_1}} \frac{|S'_{Q_2}\cap J|}{|J|}\\ &\asymp \Psi^\pdim(Q_1) \cdot \max_{J\in \II_{Q_1}} \frac{|S'_{Q_2}\cap J|}{|J|}\,.\label{SQ1SQ2}
\end{align}
By Subclaim \ref{subclaimsize}, we have $|J| \asymp \psi(Q_1)$ for all $J\in \II_{Q_1}$. Now fix $J\in \II_{Q_1}$, and we will estimate $|S'_{Q_2}\cap J|$. The calculation  proceeds differently depending on the following three ``ranges'':
\begin{itemize}
\item[1.] If $\psi(Q_1) \leq \phi(Q_2)$ (i.e. $Q_1,Q_2$ are close to each other), then by Subclaim~\ref{subclaimseparation} (with $Q = Q_2$) and Subclaim~\ref{subclaimsize} (with $Q = Q_1$), we have that $J$ intersects only boundedly many elements of $\II_{Q_2}$, and thus by Subclaim \ref{subclaimsize} (with $Q = Q_2$) we otain that $|S'_{Q_2}\cap J| \lesssim \psi(Q_2)$. Applying \eqref{SQ1SQ2} and Subclaim \ref{subclaimsize} (with $Q = Q_1$) yields
\[
\frac{|S'_{Q_1}\cap S'_{Q_2}\cap I|}{|I|} \lesssim \Psi^\pdim(Q_1) \frac{\psi(Q_2)}{\psi(Q_1)} = \Psi^{\pdim-1}(Q_1) \Psi(Q_2) \left(\frac{Q_1}{Q_2}\right)^{\qdim/\pdim}.
\]
Since $\Psi$ is a Hardy $L$-function such that $\Psi(Q)\to 0$ as $Q\to\infty$, $\Psi$ is eventually decreasing and thus $\Psi(Q_2) \leq \Psi(Q_1)$, so
\[
\frac{|S'_{Q_1}\cap S'_{Q_2}\cap I|}{|I|} \lesssim \Psi^{\pdim}(Q_1) \left(\frac{Q_1}{Q_2}\right)^{\qdim/\pdim}.
\]
\item[2.] If $\phi(Q_2) < \psi(Q_1) \leq \theta(Q_2)$ (i.e. $Q_1,Q_2$ are at medium distance to each other), then by Subclaim \ref{subclaimseparation} (with $Q = Q_2$), we have that $J$ intersects $\lesssim \psi(Q_1)/\phi(Q_2)$ elements of $\II_{Q_2}$, and thus by Subclaim \ref{subclaimsize} (with $Q = Q_2$) we have $|S'_{Q_2}\cap J| \lesssim \psi(Q_1)\psi(Q_2)/\phi(Q_2)$. Applying \eqref{SQ1SQ2} and \eqref{phidef} yields
\[
\frac{|S'_{Q_1}\cap S'_{Q_2}\cap I|}{|I|} \lesssim \Psi^\pdim(Q_1) \frac{\psi(Q_2)}{\phi(Q_2)} \asymp \Psi^\pdim(Q_1) \frac{1}{|\log\Psi(Q_2)|} \cdot
\]
\item[3.] If $\psi(Q_1) > \theta(Q_2)$ (i.e. $Q_1,Q_2$ are far from each other), then by Subclaims~\ref{subclaimsize} and  \ref{subclaimdistribution}, we have that $|S'_{Q_2}\cap J| \lesssim \Psi^\pdim(Q_2)\cdot |J|$ for all $J \in \II_{Q_1}$. Applying \eqref{SQ1SQ2} yields
\[
\frac{|S'_{Q_1}\cap S'_{Q_2}\cap I|}{|I|} \lesssim \Psi^\pdim(Q_1) \Psi^\pdim(Q_2).
\]
\end{itemize}
We now proceed to estimate the sum
\begin{equation}\label{vb4.21}
\sum_{\substack{Q_1 \leq Q_2 \\ Q_1,Q_2 \in \QQ'}} \frac{|S'_{Q_1} \cap S'_{Q_2}\cap I|}{|I|}
\end{equation}
where $\QQ' = \QQ\cap [Q^-,Q^+]$ for some large $Q^-,Q^+ \in \QQ$. Each term in this sum can be considered to belong to exactly one of the three cases above. Thus, we can write \eqref{vb4.21} as the sum of three smaller series $\Sigma_1,\Sigma_2,\Sigma_3$ corresponding to the terms from each of the three cases:
\[
\Sigma_k = \sum_{\substack{Q_1 \leq Q_2 \\ Q_1,Q_2 \in \QQ' \\ \text{Case $k$}}} \frac{|S'_{Q_1} \cap S'_{Q_2}\cap I|}{|I|}    \ \ \ \  \qquad (k=1,2,3)  \, ,
\]
Clearly,
\begin{equation}
\label{sigma3}
\Sigma_3 \lesssim \left(\sum_{Q\in \QQ'} \Psi^\pdim(Q) \right)^2
\end{equation}
and
\begin{equation}
\label{sigma1}
\Sigma_1 \lesssim \sum_{\substack{Q_1 \leq Q_2 \\ Q_1,Q_2 \in \QQ'}} \Psi^\pdim(Q_1) \left(\frac{Q_1}{Q_2}\right)^{\qdim/\pdim}
\asymp \sum_{Q\in \QQ'} \Psi^\pdim(Q).
\end{equation}
To bound $\Sigma_2$, fix $Q_1\in \QQ$, and let $Q_2^-,Q_2^+\in \QQ$ be the smallest and largest elements of $\QQ$ such that the pairs $(Q_1,Q_2^\pm)$ fall into Case 2, i.e. such that $\phi(Q_2^\pm) < \psi(Q_1) \leq \theta(Q_2^\pm)$. Then $\phi(Q_2^-) < \theta(Q_2^+)$, or equivalently
\[
\left(\frac{Q_2^+}{Q_2^-}\right)^{\qdim/\pdim} < \frac{\Theta(Q_2^+)}{\Phi(Q_2^-)}\cdot
\]
Fix $0 < \epsilon < \qdim/\pdim$. Since $\Theta$ is a Hardy $L$-function such that $1 \leq \Theta(Q) \leq Q^{\epsilon/2}$ for all sufficiently large $Q$, if $Q^-$ is sufficiently large, then we have that
\[
\Theta(\Lambda Q) \leq \Lambda^\epsilon \Theta(Q) \;\;\all \Lambda \geq 1,\; Q \geq Q^-   \, .
\]
Hence, since $\Phi$ is decreasing, we have that
\[
\frac{\Theta(Q_2^+)}{\Phi(Q_2^-)} \leq \left(\frac{Q_2^+}{Q_2^-}\right)^\epsilon \min_{Q\in \QQ'(Q_1)}\frac{\Theta(Q)}{\Phi(Q)}
\]
and thus
\[
\left(\frac{Q_2^+}{Q_2^-}\right)^{n/m-\epsilon} < \min_{Q\in \QQ'(Q_1)}\frac{\Theta(Q)}{\Phi(Q)}
\]
where $\QQ'(Q_1) = \QQ\cap [Q_2^-,Q_2^+]$. It follows that
\[
\#(\QQ'(Q_1)) = 1 + \log_2\left(\frac{Q_2^+}{Q_2^-}\right) \lesssim 1 + \min_{Q_2\in \QQ'(Q_1)}\log\left(\frac{\Theta(Q_2)}{\Phi(Q_2)}\right)
\]
and thus
\begin{align*}
\Sigma_2 &\lesssim \sum_{Q_1 \in \QQ'} \Psi^\pdim(Q_1) \sum_{Q_2 \in \QQ'(Q_1)} \frac{1}{|\log\Psi(Q_2)|} \\[1ex] &\lesssim \sum_{Q_1 \in \QQ'} \Psi^\pdim(Q_1) \max_{Q_2\in \QQ'(Q_1)}\frac{\log\left(\frac{\Theta(Q_2)}{\Phi(Q_2)}\right)}{|\log\Psi(Q_2)|}\cdot
\end{align*}
It follows from \eqref{Deltabound}, \eqref{thetadef2}, and \eqref{phidef} that
\[
\log\left(\frac{\Theta(Q)}{\Phi(Q)}\right) \lesssim |\log\Psi(Q)|
\]
and thus
\begin{equation}
\label{sigma2}
\Sigma_2 \lesssim \sum_{Q\in \QQ'} \Psi^\pdim(Q).
\end{equation}
Combining \eqref{sigma3}, \eqref{sigma1}, and \eqref{sigma2} finishes the proof of \eqref{quasiindep}.
\end{proof}

\medskip

\subsection{Proof of Theorem~\ref{theorem3}}

In view of Remarks \ref{remark1} and \ref{remark2}, proving the following claim will complete the proof of Theorem~\ref{theorem3}.

\begin{claim}
\label{claim3}
Suppose that $\bfalpha$ is nonsingular and that $\psi = c\psi_\noverm$ for some $0 < c \leq 1$. Then $|\ff^{-1}(\Twist_\psi(\bfalpha))| = |I_0|$.
\end{claim}

\begin{proof}
Let $\Delta\equiv\epsilon \leq 1$ be as in Lemma~\ref{lemmanonsingular}, and let $\eta(Q) = Q^\epsilon$ for some small $0 < \epsilon < n/2m$. As in the proof of Claim \ref{claim2} we let $\Theta$ and $\theta$ be defined by \eqref{thetadef2} and \eqref{thetadef3}, and we note that $\eta(Q) \to \infty$ as before, and that $Q^{-\qdim/\pdim} \leq \theta(Q) = o(\psi^{1/2}(Q))$ for all sufficiently large $Q$. Thus by Lemmas \ref{lemmanonsingular} and \ref{lemmasummary}, for infinitely many $Q\in\QQ$, \eqref{BQbound} holds; that is
$$|B_Q| \lesssim \eta^{-1/(m-1)}(Q) = o(1) \, . $$ Fix $I \subset I_0$. Repeating the proof of \eqref{SQprimelowerbound} shows that $|S_Q \cap I| \gtrsim \Psi^\pdim(Q) \cdot |I|$ for any such $Q$. But since $\Psi \equiv c \asymp 1$, this shows that $|S_Q \cap I| \gtrsim |I|$ for any such $Q$; taking the limsup with respect to $Q$ gives $|\ff^{-1}(\TTwist_\psi(\bfalpha)) \cap I| \gtrsim |I|$. Since $I$ was arbitrary, Lemma~\ref{lemmafullmeasure} implies that $\ff^{-1}(\TTwist_\psi(\bfalpha))$ has full measure in $I_0$.
\end{proof}

\section{Proof of Theorem~\ref{theorem5}}

We start by establishing the ``Dani correspondence'' for very singular matrices -- in general the correspondence  connects problems in Diophantine approximation to the behaviour of flows on homogeneous spaces.

\begin{lemma}[Dani correspondence principle]
\label{lemmadani2}
Let $\bfalpha$ be a very singular $m\times n$ matrix. Then there exists $\delta > 0$ such that for all $Q$ sufficiently large,
\[
\lambda_1(\Lambda_Q^*) \leq Q^{-\delta}.
\]
\end{lemma}
\begin{proof}
Let $\epsilon > 0$ and $Q_\epsilon \geq 1$ be as in \eqref{verysingulardef}. Fix $Q \geq Q_\epsilon$, and choose $(\pp,\qq)\in \Z^m\times \Z^n$ such that \eqref{verysingulardef} holds. Let $\w Q = Q^{1+\delta}$, where $\delta>0$ is small. Let $\rr = g_{\w Q} u_\bfalpha (\pp,\qq) \in \Lambda_{\w Q}$. Then
\begin{align*}
\|\rr\| &\leq  \max(\w Q^{n/m} \|\bfalpha \cdot \qq + \pp\| , \w Q^{-1} \|\qq\|)\\[2ex]
&\leq \max(Q^{(1+\delta)n/m} Q^{-(n/m+\epsilon)}, Q^{-(1+\delta)} Q)\\[2ex]
&= \max(Q^{\delta n/m - \epsilon},Q^{-\delta}).
\end{align*}
By choosing $\delta$ small enough we can guarantee that $\delta n/m - \epsilon \leq -\delta$ and thus $\|\rr\| \lesssim Q^{-\delta}$. Thus $\lambda_1(\Lambda_Q) \lesssim Q^{-\delta}$, and by Minkowski's second theorem we have $\lambda_d(\Lambda_Q) \gtrsim Q^{\delta/(d-1)}$. By the duality principle for Minkowski minima, it follows that $\lambda_1(\Lambda_Q^*) \lesssim Q^{-\delta/(d-1)}$.
\end{proof}

Now fix $m\geq 2$, $n\geq 1$, an $m\times n$ matrix $\bfalpha$, a curve $\CC \subset \R^m$, and a $C^m$ parameterization $\ff:I_0\to\CC$ with Wronskian bounded from below. We want to show that $\ff^{-1}(\Tad(\bfalpha))$ has full measure in $I_0$. Note that $\Tad(\bfalpha) \supset \R^m\butnot \Twist_{\epsilon\psi_\noverm}(\bfalpha)$ for all $\epsilon > 0$, where $\psi(q) := \psi_\noverm(q) := q^{-n/m}$. Thus by Lemma~\ref{lemmadani}, we have
\[
\Tad(\bfalpha)  \  \supset  \  \R^m\butnot \TTwist_{\psi_\noverm}(\bfalpha)\  =  \ \R^m\butnot \limsup_{\QQ\ni Q\to\infty} A_{\psi_\noverm,Q}.
\]
For convenience let $A_Q = A_{\psi_\noverm,Q}$. By the Borel--Cantelli lemma, if
\begin{equation}
\label{ETS5}
\sum_{Q\in\QQ} |\ff^{-1}(A_Q)| < \infty,
\end{equation}
then $\ff^{-1}(\Tad(\bfalpha))$ has full measure in $I_0$, and we are done. To demonstrate \eqref{ETS5}, fix $Q\in\QQ$. By Lemma~\ref{lemmadani2}, if $Q$ is large enough then there exists $\rr \in\Lambda_Q^*\butnot\{\0\}$ such that $\|\rr\| \leq Q^{-\delta}$. Write
\[
\rr = ((g_Q u_\bfalpha)^T)^{-1}(\pp,\qq) = (Q^{-n/m} \pp, Q (\qq - \bfalpha^T \cdot\pp)).
\]
If $\pp = \0$, then $\|\rr\| = Q\|\qq\| \geq Q$, which contradicts $\|\rr\| \leq Q^{-\delta}$. So $\|\pp\| \geq 1$. Now if $\bfbeta\in A_Q$, then there exist $\ss\in \Lambda_Q$ and $\tt\in \RR(\Psi(Q),1) = \RR(1,1)$ such that $(Q^{\qdim/\pdim}\bfbeta,\0) = \ss + \tt$, and thus
\[
\pp\cdot \bfbeta = \rr\cdot (\ss + \tt) \in \Z + B(0,2\|\rr\|) \subset \NN(\Z,2 Q^{-\delta}).
\]
If $t\in \ff^{-1}(A_Q)$, then
\[
g(t) \df \pp\cdot \ff(t) \in \NN(\Z,2 Q^{-\delta}).
\]
In other words,
\[
\ff^{-1}(A_Q) \subset \bigcup_{k\in\Z} g^{-1}\big(B(k,2 Q^{-\delta})\big).
\]
By Lemma~\ref{lemmanondegenerate}, the measure of any single term in the union on the right-hand side is
\[
|g^{-1}\big(B(k,2 Q^{-\delta})\big)|
\lesssim (Q^{-\delta}/\|\pp\|)^{1/(m-1)} \leq Q^{-\delta/(m-1)}.
\]
Now let $I \subset I_0$ be an interval of monotonicity for $g'$. By replacing $g$ by $t \mapsto \pm g(\pm t)$, we may without loss of generality suppose that $g'$ is positive and increasing on $I$. Let $k_0\in\Z$ be the smallest integer such that $g(I)\cap B(k_0,2Q^{-\delta})\neq \emptyset$. Then $(g^{-1})'$ is positive and decreasing on $g(I)$, so \medskip
\begin{align*}
|I\cap &\ \ff^{-1}(A_Q)| \leq  \ \big|I\cap g^{-1}\big(\NN(\Z,2Q^{-\delta})\big)\big|= \\[1ex]
&\hspace*{4ex}\text{(change of variables)}\\
&=  \ |I\cap g^{-1}(B(k_0,2Q^{-\delta}))| \ + \ \sum_{k>k_0} \int_{g(I) \cap B(k,2Q^{-\delta})}  (g^{-1})' \\
&\hspace*{4ex}\text{(since $(g^{-1})'$ is decreasing)}\\
&\leq  \ |I\cap g^{-1}(B(k_0,2Q^{-\delta}))| \ +  \ \sum_{k > k_0} 4Q^{-\delta}  \int_{g(I) \cap [k-1+Q^{-\delta},k+Q^{-\delta}]} (g^{-1})'  \\
&\leq \  |I\cap g^{-1}(B(k_0,2Q^{-\delta}))| \ +  \ 4Q^{-\delta}  \int_{g(I)} (g^{-1})' \\
&\hspace*{4ex}\text{(by Lemma~\ref{lemmanondegenerate})}\\
&\lesssim \  Q^{-\delta/(m-1)}  \ +  \ Q^{-\delta}|I| \ \lesssim  \ Q^{-\delta/(m-1)}.
\end{align*}
Since $\ff$ is nondegenerate, $I_0$ can be covered by a collection of intervals on which $g'$ is monotonic and whose cardinality is bounded independent of $\pp$. Thus
\[
|\ff^{-1}(A_Q) | \lesssim Q^{-\delta/(m-1)}.
\]
This proves \eqref{ETS5}.

\appendix
\section{Proof of Theorem~\ref{theorem4}}
\label{sectionappendix}
We define the absolute game as introduced by McMullen in \cite{McMullen_absolute_winning}.

Let $\Lambda\subset \R^d$ be a closed set. For each $0 < \beta < 1$, the \emph{absolute $\beta$-game on $\Lambda$} is an infinite game played by two players, Alice and Bob, who take turns choosing balls $B_1,A_1,B_2,A_2,\ldots$ in $\R^d$ with centers in $\Lambda$, with Bob moving first. The players must choose their moves so as to satisfy the relations
\begin{equation}
\label{absolute}
B_{k + 1} \subset B_k \setminus A_k
\end{equation}
and
\[
\rho(A_k) \leq \beta \rho(B_k) \text{ and } \rho(B_{k + 1}) \geq \beta \rho(B_k)   \ \text{ for } k\in\N,
\]
where $\rho(B)$ denotes the radius of a ball $B$. Due to condition \eqref{absolute} we think of Alice as ``deleting'' her chosen ball $A_k$, whereas Bob is thought of as ``moving into'' his choice $B_k$. The completeness of $\Lambda$ implies that the intersection $\bigcap_k B_k$ is a singleton, say $\bigcap_k B_k = \{\xx^{(\infty)}\}$, and the point $\xx^{(\infty)}\in \Lambda$ is called the \emph{outcome} of the game. A set $S \subseteq \Lambda$ is said to be \emph{absolute $\beta$-winning on $\Lambda$} if Alice has a strategy guaranteeing that the outcome lies in $S$, regardless of the way Bob chooses to play. It is said to be \emph{absolute winning on $\Lambda$} if it is absolute $\beta$-winning for every $0 < \beta < 1$. A fundamental property of absolute winning sets is the following:

\begin{lemma}[{\cite[Lemmas 3 and 4]{BGSV}}]
\label{lemmaBFKRW}
Let $S$ be an absolute winning set on a closed set $\Lambda\subset \R^d$, and let $K \subset \Lambda$ be a closed Ahlfors $\delta$-regular set. Then
\[
\HD(K\cap S) = \HD(K) = \delta.
\]
\end{lemma}

Since manifolds are Ahlfors regular, it follows that to prove Theorem~\ref{theorem4}, it suffices to show that the set of badly $\bfalpha$-approximable vectors is absolute winning on $\R^m$. To prove this, we need the following lemma:

\begin{lemma}
Fix $0 < \epsilon < 1$, and let $\bfalpha$ be an $m\times n$ matrix such that
\begin{equation}
\label{epsilondef}
\|\bfalpha\cdot \qq + \pp\| > \epsilon \|\qq\|^{-n/m}
\end{equation}
for all $(\pp,\qq)\in \Z^m\times(\Z^n\butnot\{\0\})$. Then for all $Q \geq 1$ and $\bfbeta\in\R^m$, there is at most one pair $(\pp,\qq) \in \Z^m\times \Z^n$ such that
\[
\|\bfalpha\cdot \qq + \pp + \bfbeta\| \leq \tfrac12 \epsilon Q^{-n/m}  \quad \text{ and }    \quad \|\qq\| \leq \tfrac12 Q.
\]
\end{lemma}
\begin{proof}
Suppose that there were two such pairs, $(\pp_1,\qq_1)$ and $(\pp_2,\qq_2)$. Let $\pp = \pp_2 - \pp_1$ and $\qq = \qq_2 - \qq_1$. Then
\[
\|\bfalpha\cdot \qq + \pp\| \leq \|\bfalpha\cdot \qq_1 + \pp_1 + \bfbeta\| + \|\bfalpha\cdot \qq_2 + \pp_2 + \bfbeta\| \leq \epsilon Q^{-n/m} 
\]
and
\[
\|\qq\| \leq \|\qq_1\| + \|\qq_2\| \leq Q,
\]
from which it follows that
$\|\bfalpha\cdot \qq + \pp\| \leq \epsilon \|\qq\|^{-n/m}$.
By our assumption on $\bfalpha$, it follows that $\qq = \0$. But then
\[
\|\pp\| = \|\bfalpha\cdot \qq + \pp\| \leq \epsilon < 1
\]
and thus $\pp = \0$, so $(\pp_1,\qq_1) = (\pp_2,\qq_2)$.
\end{proof}

Now let $\bfalpha$ be a badly approximable $m\times n$ matrix. Then there exists $0 < \epsilon < 1$ such that \eqref{epsilondef} holds for all $(\pp,\qq)\in \Z^m\times(\Z^n\butnot\{\0\})$. Now we give a strategy for Alice to win the absolute $\beta$-game on $\R^m$ as follows: if Bob makes the move
\[
B_k = B(\xx^{(k)},\rho_k),
\]
then let
\[
Q_k = \left(\frac{\epsilon}{4\rho_k}\right)^{m/n},
\]
and let $(\pp^{(k)},\qq^{(k)})\in \Z^m\times\Z^n$ be the unique pair such that
\[
\|\bfalpha\cdot \qq^{(k)} + \pp^{(k)} + \xx^{(k)}\| \leq 2\rho_k = \tfrac12 \epsilon Q_k^{-n/m} \text{ and } \|\qq^{(k)}\| \leq \tfrac12 Q_k
\]
if such a pair exists, and $(\pp^{(k)},\qq^{(k)}) = (\0,\0)$ otherwise. Then we let Alice's next move be $A_k = B(\yy^{(k)},\beta\rho_k)$, where
\[
\yy^{(k)} = -(\bfalpha\cdot\qq^{(k)} + \pp^{(k)}).
\]
To show that the outcome $\xx^{(\infty)}$ is badly $\bfalpha$-approximable, fix $(\pp,\qq)\in\Z^m\times\Z^n$ such that $\|\qq\| \geq \frac12 Q_1$, and let $k$ be chosen so that $\frac12 Q_{k-1} \leq \|\qq\| <  \frac12 Q_k$. If
\[
\|\bfalpha\cdot \qq + \pp + \xx^{(k)}\| \leq 2\rho_k,
\]
then $(\pp,\qq) = (\pp^{(k)},\qq^{(k)})$, and thus $\xx^{(\infty)} \notin B(\yy^{(k)},\beta \rho_k)$. So
\[
\|\bfalpha\cdot \qq + \pp + \xx^{(\infty)}\| = \|\xx^{(\infty)} - \yy^{(k)}\| \geq \beta \rho_k.
\]
On the other hand, if $\|\bfalpha\cdot \qq + \pp + \xx^{(k)}\| > 2\rho_k$, then
\[
\|\bfalpha\cdot \qq + \pp + \xx^{(\infty)}\| \geq \|\bfalpha\cdot \qq + \pp + \xx^{(k)}\| - \|\xx^{(\infty)} - \xx^{(k)}\| \geq 2\rho_k - \rho_k = \rho_k.
\]
So either way we have
\[
\|\bfalpha\cdot \qq + \pp + \xx^{(\infty)}\| \geq \beta \rho_k \geq \beta^2 \rho_{k-1} = \tfrac14 \epsilon \beta^2 Q_{k-1}^{-n/m} \geq  \tfrac14 \epsilon \beta^2 2^{-n/m} \|\qq\|^{-n/m}.
\]
This shows that $\xx^{(\infty)}$ is badly $\bfalpha$-approximable and we are done.

\bigskip
\bigskip

\noindent{\it Acknowledgements.} VB was supported by an EPSRC grant, EP/Y016769/1. DS was supported by a Royal Society University Research Fellowship, URF\textbackslash{}R1\textbackslash{}180649.
SV would like to congratulate Barak Weiss on reaching middle age with such grace, dignity, and modesty: in a world too often dominated by loco and power-hungry figures, I have the deepest respect and admiration for the way you carry yourself — with integrity, thoughtfulness, and, above all, an open mind.

\bibliographystyle{abbrv}

\bibliography{bibliography}

\end{document}